\documentclass[10pt]{article}
\usepackage[top=20mm, bottom=20mm, left=21mm, right=21mm]{geometry}
\usepackage{graphicx}
\usepackage{amsmath,amsthm,amssymb,amsfonts}
\usepackage{bm}
\usepackage{lipsum}
\usepackage{textgreek}

\theoremstyle{plain}
\newtheorem{theorem}{Theorem}[section]

\theoremstyle{definition}
\newtheorem{definition}{Definition}[section]

\newenvironment{assumption}
{\par\medskip\noindent}
{\par\medskip}
\makeatletter
\newcommand{\assumptionlabel}[2]{%
	\def\@currentlabel{#1}%
	\label{#2}%
}
\makeatother

\usepackage{xcolor}
\usepackage{soul}

\usepackage[utf8]{inputenc}
\usepackage[T1]{fontenc}
\usepackage{mathrsfs}
\usepackage{natbib}

\usepackage[colorlinks=true, allcolors=cyan]{hyperref}
\usepackage{comment}

\usepackage{mathtools}

\newcommand{\om}{\Omega}

\newcommand{\vp}{\bm{\varphi}}

\newcommand{\bv}{\bm{v}}

\newcommand{\be}{\bm{E}}

\newcommand{\bff}{\bm{f}}
\newcommand{\bF}{\bm{F}}
\newcommand{\bg}{\bm{G}}

\newcommand{\bi}{\bm{I}}

\newcommand{\dv}{\nabla\cdot}

\newcommand{\bu}{\bm{u}}

\newcommand{\ba}{\bm{a}}
\newcommand{\bb}{\bm{b}}

\newcommand{\bd}{\bm{d}}

\newcommand{\bA}{\bm{A}}
\newcommand{\bB}{\bm{B}}

\newcommand{\bD}{\bm{D}}

\newcommand{\bM}{\bm{M}}

\newcommand{\bW}{\bm{W}}

\newcommand{\bZ}{\bm{Z}}
\newcommand{\bJc}{\bm{\mathcal{J}}}

\newcommand{\som}{S \times \Omega}

\newcommand{\ino}{\int_{\Omega}}
\newcommand{\ins}{\int_{S}}
\newcommand{\inso}{\int_{S\times\Omega}}

\newcommand{\ddt}{\frac{d}{dt}}

\newcommand{\ltwo}{L^2}
\newcommand{\bDv}{\bm{D}(\bv)}

\newcommand{\bWv}{\bm{W}(\bv)}

\def\dev{\mathrm{dev}}

\newcommand{\bvne}{\bm{v}_n^{\varepsilon}}
\newcommand{\bFne}{\bm{F}_n^{\varepsilon}}
\newcommand{\qne}{q_n^{\varepsilon}}

\newcommand{\bbS}{\mathbb{S}}

\def\bPS{{\bm\Psi}}
\def\bDc{{\bm{\mathcal{D}}}}

\usepackage{etoolbox}
\usepackage{dsfont}
\usepackage{bbm}
\usepackage{tikz}
\usetikzlibrary{calc}
\usetikzlibrary{arrows.meta}
\usetikzlibrary{patterns}
\usepackage[normalem]{ulem}
\usepackage{subcaption}
\usepackage{caption}
\usepackage{chngcntr}
\counterwithout{figure}{section}
\usepackage{todonotes}
\usepackage{amsopn}
\usepackage{fix-cm}
\usepackage{ifthen}
\usepackage{authblk}   

\numberwithin{equation}{section}

\def\O{\Omega}

\def\G1{{\bf 1}}

\def\sc{\longrightarrow}
\def\f{{{\bm f}}}
\def\e{\varepsilon}

\definecolor{navyblue}{rgb}{0.0, 0.0, 0.5}

\newcommand{\tr}{\mathrm{tr}}
\newcommand{\DDt}{\frac{D}{Dt}}

\def\R{{\mathbb{R}}}

\def\p{\partial}

\def\Ec{{\mathcal{E}}}

\def\Qc{{\mathcal{Q}}}

\def\Uc{{\mathcal{U}}}

\def\Vc{{\mathcal{V}}}

\def\wc{{\;\rightharpoonup\;}}

\def\X{\times }
\def\div{{\mathrm{div}}}

\def\G{{\mathbb{G}}}

\definecolor{skyblue}{RGB}{135,206,235}
\definecolor{deepskyblue}{RGB}{0, 191, 255}
\usepackage{nicefrac}
\usepackage{mathtools}

\usepackage{charter} 
\usepackage{longtable}
\usepackage{float}
\usepackage{supertabular}

\usepackage{overpic}
\usepackage[section]{placeins}

\usepackage{enumitem}
\usepackage{siunitx}

\theoremstyle{plain}
\newtheorem{proposition}[theorem]{Proposition}
\newtheorem{lemma}[theorem]{Lemma}

\theoremstyle{definition}
\newtheorem{remark}[theorem]{Remark}

\def\bX{{\bm X}}
\def\bF{{\bm F}}
\def\bE{{\bm E}}
\def\bI{{\bm I}}

\definecolor{DarkRed}{RGB}{220,0,0}
\usepackage{bookmark}

\hypersetup{
	colorlinks=true,
	linkcolor=DarkRed,
	filecolor=magenta,
	urlcolor=blue!60!black,
	citecolor=blue!60!black,
	pdftitle={The Vanishing-Diffusion Limit of an Incompressible Visco-Morphoelastic System: Weak Solutions and a Jaumann Defect},
	pdfauthor={S. Banerjee, A. Chakrabortty, H. S. Mahato, G. P. Raja Sekhar},
}

\newcommand{\keywords}[1]{%
	\par\medskip\noindent\textbf{Keywords.}\ #1}
\newcommand{\msc}[1]{%
	\par\smallskip\noindent\textbf{2020 Mathematics Subject Classification.}\ #1}

\title{\bfseries The Vanishing-Diffusion Limit of an Incompressible Visco-Morphoelastic System: Weak Solutions and a Jaumann Defect}
\author{Swarupananda Banerjee\thanks{%
		Department of Mathematics, Indian Institute of Technology
		Kharagpur, Kharagpur 721302, West Bengal, India.
		Email: \href{mailto:sbanerjee.24@kgpian.iitkgp.ac.in}{\texttt{sbanerjee.24@kgpian.iitkgp.ac.in}}},~ 
	Amartya Chakrabortty\thanks{ 
		Department of Mathematics, Indian Institute of Technology
		Kharagpur, Kharagpur 721302, West Bengal, India.
		Email: \href{mailto:amartyat26@kgpian.iitkgp.ac.in}{\texttt{amartyat26@kgpian.iitkgp.ac.in}},~ \href{mailto:amartya.chakrabortty@gmail.com}{\texttt{amartya.chakrabortty@gmail.com}}},~
	Hari Shankar Mahato\thanks{%
		Department of Mathematics, Indian Institute of Technology
		Kharagpur, Kharagpur 721302, West Bengal, India.
		Email: \href{mailto:hsmahato@maths.iitkgp.ac.in}{\texttt{hsmahato@maths.iitkgp.ac.in}}},~ and~ 
	Raja Sekhar G P\thanks{%
		Department of Mathematics, Indian Institute of Technology
		Kharagpur, Kharagpur 721302, West Bengal, India.
		Email: \href{mailto:rajas@iitkgp.ac.in}{\texttt{rajas@iitkgp.ac.in}}}}

\date{}

\begin{document}
	
	\maketitle
	
	\vspace{-12mm}
	\begin{abstract}
		We study an incompressible visco-morphoelastic system in three dimensions
		in which the Eulerian effective strain is transported by the
		Zaremba--Jaumann rate and coupled to a Navier--Stokes momentum balance
		through a Kelvin--Voigt stress. The strain remains symmetric and splits
		into a scalar transport--relaxation equation for its trace and a
		deviatoric equation containing the nonlinear stretching and Jaumann
		terms. We regularize both equations by diffusion of order $\e$, which admits a mechanical interpretation as weakly nonlocal remodeling with a characteristic length scale. For fixed $\e>0$ we construct Leray--Hopf type weak
		solutions by a truncated Galerkin scheme and remove the truncation
		a posteriori using a maximum principle for the scalar variable. As
		$\e\to0$, the scalar variable converges strongly in $L^p((0,T)\times\om)$ for every $1\le p<\infty$, by a DiPerna--Lions commutator
		argument on bounded Lipschitz domains, allowing the stretching term
		to be identified in the limit. In contrast, the Jaumann commutator is
		a product of two weakly convergent sequences and its limit cannot be
		identified from the available estimates. We therefore obtain a defect
		\[
		\bDc\in L^2(0,T;(H^2(\om;\bbS_0))^\ast)
		\]
		in the limiting deviatoric equation. We give two sufficient conditions
		for $\bDc=0$. The energy-level cancellation used in related
		corotational viscoelastic models does not close here because the
		coefficient of the stretching term is itself an unknown.
	\end{abstract}
	
	\msc{35Q35, 35B25, 76A10, 74L15, 35D30.}

	{\small	\tableofcontents}

	\section{Introduction}
	\label{sec:intro}
	
	
	Morphoelasticity describes growth-induced permanent deformations in
	biological tissues.
	Originating from the multiplicative decomposition of finite
	plasticity~\cite{lee1969}, the framework was introduced into
	biomechanics by Rodriguez, Hoger, and McCulloch~\cite{rodriguez1994}
	and subsequently developed in the continuum theory of biological
	growth~\cite{goriely2017,goriely2011,hoger1993}. Morphoelasticity
	distinguishes the deformation associated with growth and remodeling
	from the elastic deformation required to accommodate these changes.
	Since the stress-free configuration generated by growth need not be
	compatible, this framework naturally accounts for growth-induced
	residual stresses and permanent deformation
	\cite{ambrosi2002,rodriguez1994}. It has been applied to a range of
	biological problems, including tissue remodeling, wound contraction,
	and post-burn contraction
	\cite{egberts2021,koppenolVermolenNiessen2017}.
	
	Classical morphoelasticity, however, primarily describes an elastic
	response coupled to an evolving stress-free configuration. In
	biological tissues and cellular aggregates, growth and remodeling
	occur concurrently with dissipative processes such as cellular
	rearrangement, adhesion remodeling, and cell division, which may
	induce time-dependent stress relaxation
	\cite{AmbrosiPreziosi2009,Marmottant2009}. This has motivated
	extensions of morphoelasticity incorporating viscoelastic relaxation.
	Ambrosi and Preziosi~\cite{AmbrosiPreziosi2009} introduced an
	elasto-visco-growth framework in which stress relaxation is associated
	with cellular rearrangement, while related morpho-viscoelastic models
	have been developed for tissue remodeling
	\cite{KoppenolVermolen2017}. More recently, Bandil and
	Vernerey~\cite{BandilVernerey2024} proposed a multiscale
	morpho-viscoelastic framework for proliferating cell aggregates.
	These developments motivate the study of morphoelastic systems in
	which growth-induced remodeling is coupled to time-dependent
	mechanical relaxation.
	
	From a mathematical perspective, rigorous analysis of such coupled
	systems remains limited. Existing studies of morphoelastic and
	visco-morphoelastic models have largely focused on model derivation,
	biological applications, numerical simulation, and stability analysis,
	often in reduced dimensions or in quasi-static settings
	\cite{Asghar2025,egberts2021,egberts2023}. By contrast, rigorous
	weak-existence theories are available for several related viscoelastic
	fluid and solid models, including Oldroyd-type fluids
	\cite{BarrettBoyaval2011,KremlPokornySalom2015,LionsMasmoudi2000} and nonlinear viscoelastic solids
	\cite{BulicekPatelSuliSengul2022}. These theories, however, do not
	incorporate the growth-induced strain evolution and nonlinear coupling
	characteristic of the present visco-morphoelastic setting. In
	particular, a rigorous theory for fully dynamic three-dimensional
	models of the type considered here, including weak existence and
	singular limits associated with vanishing mechanical regularization,
	appears to be absent.
	
	\smallskip
	We consider the incompressible visco-morphoelastic system obtained by
	coupling the effective-strain evolution of Hall~\cite[Eq.\,(5.64)]{Hall2008} to the
	balance of linear momentum with a Kelvin--Voigt stress. Let
	$\om\subset\R^3$ be a bounded Lipschitz domain and let $S=(0,T)$. The
	Eulerian effective strain $\be$ satisfies
	\begin{equation}\label{eq:intro-strain}
		\p_t\be+(\bv\cdot\nabla)\be
		+\be\,\bWv-\bWv\,\be
		+\bigl(\tr(\be)-1\bigr)\bDv
		=-\bg ,
	\end{equation}
	where $\bDv$ and $\bWv$ denote the symmetric and skew-symmetric parts
	of $\nabla\bv$, respectively, and $\bg$ describes the active remodeling
	of the locally relaxed configuration. The first four terms constitute
	the Zaremba--Jaumann derivative of $\be$, whose commutator term ensures
	objectivity. Coupling \eqref{eq:intro-strain} with the momentum balance
	gives the system \eqref{model+}. We do not model the biochemical
	processes driving remodeling explicitly and prescribe the affine law
	\begin{equation}\label{eq:intro-G}
		\bg=\alpha\be-\beta\bI ,
		\qquad \alpha>0,\quad\beta\ge0 .
	\end{equation}
	The strain-dependent term is motivated by existing morphoelastic
	formulations~\cite{KoppenolVermolen2017}, while $-\beta\bm  I$
	represents an additional strain-independent isotropic remodeling
	tendency. The law introduces no intrinsic material direction or memory.

	The analytical difficulty is already visible in
	\eqref{eq:intro-strain}. The strain is transported by the flow and
	relaxes pointwise, but the equation contains no spatial regularization.
	The natural energy estimate therefore gives $\be$ bounded only in $L^\infty(0,T;L^2(\om))$.
	At the same time, objectivity introduces the nonlinear Jaumann
	commutator $\be\bWv-\bWv\be$,
	which is pointwise orthogonal to $\be$ and hence does not contribute to
	the strain energy. 
	A further difficulty is specific to the morphoelastic coupling: the
	stretching term is multiplied by the unknown coefficient
	$\tr(\be)-1$. Since $\be$ remains symmetric, we write
	\begin{equation}\label{eq:intro-split}
		q:=\tr(\be),
		\qquad
		\bF:=\dev\be= \be - \frac{q}{3}\bI.
	\end{equation}
	The strain equation then separates into a closed scalar
	transport--relaxation equation for $q$ and a deviatoric equation for
	$\bF$ retaining the nonlinear stretching and Jaumann terms. This
	trace--deviatoric decomposition is central to the analysis.
	
	To obtain compactness at the approximation level, we introduce a diffusion of
	order $\e>0$ into both strain equations, leading to the regularized system
	\eqref{Rmodel+}. Equivalently, the remodeling law \eqref{eq:intro-G} is
	replaced by
	\[
	\bg_\e=\alpha\be-\beta\bI-\e\Delta\be ,
	\]
	together with the Neumann condition in \eqref{Rmodel+}$_6$. This law is the
	variational derivative of the quadratic remodeling potential augmented by
	the penalty $\tfrac{\e}{2}|\nabla\be|^2$, so the limit $\e\to0$ is also a
	passage from weakly nonlocal to strictly local remodeling; see
	Remark~\ref{rem:Geps}.
	
	%
	\smallskip
	
	The analysis has two stages. We first establish weak existence for the
	regularized system \eqref{Rmodel+} with fixed $\e>0$, and then study the
	limit $\e\to0$.
	
	\textbf{Existence for the regularized system.}
	For every fixed $\e>0$ we construct a Leray--Hopf type weak solution of
	\eqref{Rmodel+}. The construction is based on a truncated Galerkin
	scheme. The truncation is removed a posteriori using a maximum principle
	for the scalar variable $q^\e$, which yields an $\e$-independent
	pointwise bound on $q^\e$. In particular, the coefficient
	$q^\e-1$ in the stretching term is uniformly controlled. This is Theorem~\ref{Thm_existence}.
	
	\textbf{Vanishing-diffusion limit and strong convergence of the trace.}
	The uniform estimates allow us to extract weakly convergent subsequences
	as $\e\to0$. The strain variables, however, are controlled only in
	$L^\infty(0,T;L^2(\om))$, so weak convergence does not suffice to pass
	to the nonlinear terms in the strain equation. The trace--deviatoric
	decomposition \eqref{eq:intro-split} isolates a scalar transport--relaxation
	equation for $q^\e$. Using a DiPerna--Lions commutator argument
	\cite{DiPernaLions1989}, we recover the limiting energy identity for the
	trace and obtain
	\[
	q^\e\longrightarrow q
	\qquad\text{strongly in }L^p(\som)
	\quad\text{for every }1\le p<\infty .
	\]
	Consequently, the nonlinear stretching term
	$(q^\e-1)\bD(\bv^\e)$ is identified in the limit. This is Theorem~\ref{thm:qstrong}.
	
	\textbf{The Jaumann defect.}
	The deviatoric equation presents a different compactness
	problem. The Jaumann commutator $\bF^\e\bW(\bv^\e)-\bW(\bv^\e)\bF^\e$
	contains two factors that converge only weakly, and the available energy
	estimates provide no mechanism for identifying their product. We therefore
	record its weak limit through a defect
	\[
	\bDc\in
	L^2\bigl(0,T;(H^2(\om;\bbS_0))^*\bigr).
	\]
	All remaining terms in the deviatoric equation converge, and the limit
	system is consequently satisfied in the distributional sense with
	$\bDc$ as an additional source term. The defect is determined by the
	approximating sequence and subsequence, but in general cannot be recovered
	from the limit fields $(\bv,q,\bF)$ alone. This is the content of
	Theorem \ref{thm:defect-limit}.
	
	\textbf{Criteria for the absence of the defect.}
	Finally, we identify two sufficient conditions under which the defect
	vanishes:
	\[
	\bF^\e\to\bF
	\quad\text{strongly in }
	L^\infty(0,T;L^2(\om;\bbS_0)),\quad\text{or}\quad
	\nabla\bv^\e\to\nabla\bv
	\quad\text{strongly in }L^2(\som)^{3\X3}.
	\]
	Either condition gives strong convergence of one factor of the Jaumann
	commutator and hence identifies its limit; neither follows from the
	uniform estimates (Proposition~\ref{cor:mu_zero_implies_m_zero}).

	\smallskip
	
	The analytical structure of \eqref{model+} is closely related to
	viscoelastic fluid models in which an internal tensor is transported by
	the flow and coupled to the momentum equation. For Oldroyd-type systems,
	local strong solutions and global small-data solutions are classical
	\cite{cheminMasmoudi2001,guillopeSaut1990}, while global weak solutions
	for large data are available in the corotational case due to Lions and
	Masmoudi \cite{LionsMasmoudi2000}. Their analysis combines propagation
	of compactness with the theory of renormalized transport equations
	developed by DiPerna and Lions \cite{DiPernaLions1989}. The corresponding
	general non-corotational problem remains substantially more difficult.
	
	A common way to recover compactness is to introduce stress diffusion.
	For regularized Oldroyd-type systems, global weak existence and
	regularity results have been obtained by several authors
	\cite{BarrettBoyaval2011,ConstantinKliegl2012,
		KremlPokornySalom2015}. Vanishing-diffusion limits have also been
	studied in restricted settings
	\cite{mensah2024,wangWen2023}. Our regularization
	\eqref{Rmodel+} has the same analytical role, although here the
	diffusion acts on the morphoelastic strain rather than on a constitutive
	stress variable.
	
	The closest analytical results concern viscoelastoplastic fluids in
	which a symmetric deviatoric internal stress is transported by the
	Zaremba--Jaumann rate. Eiter, Hopf, and Mielke
	\cite{eiterHopfMielke2022} established global existence for a
	three-dimensional system with stress diffusion, while Eiter, Hopf, and
	Lasarzik \cite{eiterHopfLasarzik2023} studied the limit of vanishing
	stress diffusion. They identified the loss of weak sequential
	compactness of the Jaumann commutator and introduced
	energy-variational solutions to obtain a closed limiting solution
	concept. This approach has subsequently been developed further
	\cite{agostiLasarzikRocca2024} and applied to related multiphase
	systems \cite{chengLasarzikThomas2026}.
	
	The present problem has the same compactness obstruction. It differs in
	one respect: the stretching term in the deviatoric equation is multiplied
	by $q^\e-1$, where $q^\e$ is itself an unknown. This variable coefficient
	prevents the energy-level cancellation used in the viscoelastic setting
	(Section~\ref{sec-conclusion}) and is the reason for the trace--deviatoric
	decomposition \eqref{eq:intro-split}. We retain the distributional formulation rather
	than replacing it by an energy-variational solution concept. The loss of
	compactness is instead recorded explicitly by the defect $\bDc$, for
	which we also identify sufficient conditions for vanishing.
	
	The defect viewpoint is related to the general use of auxiliary defect
	objects to record oscillation and concentration in weak limits, beginning
	with the work of DiPerna and Majda on incompressible Euler equations
	\cite{diPernaMajda1987}. In the present setting the defect is not a
	measure but an element of
	$L^2(0,T;(H^2(\om;\bbS_0))^*)$, reflecting the weak compactness available
	for the Jaumann commutator. Its role and possible finer
	representations are discussed in Section \ref{sec-conclusion}.
	%
	
	\section{Notation and function spaces}
	Throughout the paper, let $\Omega\subset\mathbb{R}^3$ be a bounded domain with Lipschitz boundary and $S=(0,T)$.
	We denote by $L^p(\Omega)$ and $W^{k,p}(\Omega)$ the standard Lebesgue and Sobolev spaces, respectively. In particular, $H^k(\Omega)=W^{k,2}(\Omega)$,
	\[
	H_0^1(\Omega)=\{u\in H^1(\Omega):u|_{\partial\Omega}=0\},\quad\text{and}\quad H^1_{0,\div}(\Omega)^3
	:=
	\{\bu\in H_0^1(\Omega)^3:\dv\bu=0\text{ in }\Omega\}.
	\]
	Moreover,
	\[
	\mathbb{S}
	=
	\{\bM\in\mathbb{R}^{3\times3}:\bM^T=\bM\},
	\qquad
	\mathbb{S}_0
	=
	\{\bM\in\mathbb{S}:\tr(\bM)=0\}.
	\]
	Unless otherwise stated, $(\cdot,\cdot)$ denotes the standard $L^2$ inner product, with associated norm $\|\cdot\|_{L^2}$, while $\|\cdot\|_{H^1}$ denotes the norm on $H^1(\Omega)$. For a Banach space $\mathbb{X}$, its dual is denoted by $\mathbb{X}^\ast$, and $\langle\cdot,\cdot\rangle$ denotes the corresponding duality pairing, where the underlying spaces are understood from the context.
	Whenever Bochner spaces are involved, the pairing is understood between the corresponding evolution spaces, which will always be clear from the context.
	
	In addition, for any function $g(x,t)$ defined on $S\times\om$, we adopt the following shorthand for integrals:
	$$ \ino g:=\ino g(x,t)\,dx,\quad\ins g:=\ins g(x,t)\,dt,\quad \inso\,g:=\inso g(x,t)\,dx\,dt.$$
	We write $Q_t:=\O\times(0,t)$ for $0<t\le T$.

	\section{Visco-morphoelastic models in strong form}
	\label{sec-strong}
	
	\subsection{Modeling framework and incompressible reduction}
	\label{SubsecModelDerivation}
	
	We recall the continuum-mechanical origin of the system stated in
	Section \ref{sec:intro}. In morphoelasticity, the deformation gradient
	is decomposed multiplicatively as $\bF=\bA\bZ$, where $\bZ$ maps the reference
	configuration to a locally relaxed configuration and $\bA$ maps the relaxed
	configuration to the current configuration; see
	\cite{GorielyBenAmar2007,rodriguez1994}. In the infinitesimal effective-strain
	theory of Hall \cite{Hall2008}, the Eulerian effective strain is defined by
	$\be=\bI-\bA^{-1}$, and its evolution is governed by \eqref{eq:intro-strain},
	whose first four terms constitute the Jaumann derivative
	\[
	\overset{\circ}{\be}
	:=
	\p_t\be+(\bv\cdot\nabla)\be+\be\bWv-\bWv\be ,
	\]
	so that \eqref{eq:intro-strain} may be written compactly as
	$\overset{\circ}{\be}+(\tr(\be)-1)\bDv=-\bg$. Here
	\[
	\bDv:=\frac{1}{2}\bigl(\nabla \bv+\nabla \bv^\top\bigr),
	\qquad
	\bWv:=\frac{1}{2}\bigl(\nabla \bv-\nabla \bv^\top\bigr).
	\]
	A detailed derivation of \eqref{eq:intro-strain} is given in \cite{Hall2008};
	its use in a visco-morphoelastic continuum model is described in
	\cite{KoppenolVermolen2017}.
	
	Let $\rho$ denote the mass density and ${\bm \sigma}_c$ the total Cauchy stress.
	The local balances of mass and linear momentum are
	\begin{equation}
		\partial_t\rho+\div(\rho \bv)=0,
		\qquad
		\partial_t(\rho \bv)+\div(\rho \bv\otimes \bv)
		=\div{\bm \sigma}_c+\f.
		\label{EqBalanceLaws}
	\end{equation}
	Using the mass balance, the momentum equation may equivalently be written as
	\[
	\rho\bigl(\partial_t\bv+(\bv\cdot\nabla)\bv\bigr)
	=\div{\bm\sigma}_c+\f.
	\]
	Following the Kelvin--Voigt-type constitutive law used in
	\cite{KoppenolVermolen2017}, we take the viscoelastic part of the stress to be
	\begin{equation}
		{\bm \sigma}
		=
		\mu_1\bDv
		+\mu_2\div(\bv)\bI
		+\lambda \be
		+2\mu\tr(\be)\bI,
		\label{EqViscoelasticStress}
	\end{equation}
	where $\mu_1,\mu_2>0$ are viscosity coefficients and $\lambda,\mu>0$ are
	elastic coefficients.
	
	As a first analytical setting, we consider a homogeneous incompressible
	material. Thus the density is a positive constant, $\rho\equiv\rho_0>0$, and
	\[
	\div \bv=0 ,
	\]
	while the total Cauchy stress is
	\[
	{\bm \sigma}_c=-p\bI+{\bm \sigma},
	\]
	where $p$ is the pressure associated with the incompressibility constraint.
	Then the bulk-viscous term in \eqref{EqViscoelasticStress} vanishes; we
	retain it to display the viscoelastic origin of the stress.
	
	Combining \eqref{eq:intro-strain}, \eqref{EqBalanceLaws} and
	\eqref{EqViscoelasticStress} with supplementary initial and boundary
	conditions, we obtain the $3$D incompressible visco-morphoelastic system
	\begin{equation}
		\label{model+}
		\begin{aligned}
			&\rho \, \partial_t \bv + \rho\,(\bv\cdot\nabla)\bv+\nabla p=\nabla\cdot\bm{\sigma}+\f,
			&& \quad\text{in } \som ,\\[0.5em]
			&\partial_t\be+(\bv\cdot\nabla)\be
			+\be\,\bWv-\bWv\,\be 
			+\bigl(\tr(\be)-1\bigr)\,\bDv
			=-\bg ,
			&&\quad\text{in } \som ,\\[0.5em]
			&\bm{\sigma}
			=\mu_1\,\bDv
			+\mu_2\,(\nabla\cdot\bv)\,\bI
			+\lambda\be+2\mu\tr(\be)\bI,
			&&\quad\text{in } \som ,\\[0.5em]
			&\nabla\cdot\bv=0,
			&&\quad\text{in }\som,\\[0.5em]
			&\bv(\cdot,0)=\bv_0,\quad\be(\cdot,0)=\be _0, 
			&&\quad\text{in }\om,\\[0.5em]
			&\bv=\bm{0}, 
			&&\quad\text{on } S\times\partial\om.
		\end{aligned}
		\tag{\text{{\bf M1}}}
	\end{equation}

	The remodeling law $\bg$ has been left unspecified so far. Before selecting it
	we record a structural property of \eqref{model+} which is independent of that
	choice, and on which the decomposition of the next subsection rests: the
	effective strain stays symmetric whenever the remodeling law is symmetric.
	
	\begin{proposition}\label{symmetry_E_general}
		Let $(\bv,\be)$ be a sufficiently regular solution of \eqref{model+} with
		given $\be_0$ and $\bg$. If $\bg^\top=\bg$ and $\be_0=\be_0^\top$ pointwise,
		then
		\begin{equation}\label{syE}
			\be (x,t) = [\be (x,t)]^\top
			\quad\text{for all } (x,t)\in \som .
		\end{equation}
	\end{proposition}
	\begin{proof}
		From equation \eqref{model+}$_2$, we have 
		\begin{equation}\label{E}
			\DDt \be + \be \,\bWv - \bWv\,\be  + \bigl(\tr(\be ) - 1\bigr)\,\bDv= -\bg,
		\end{equation}
		where $\DDt:=\p_t+(\bv\cdot\nabla)$.
		The scalar $\tr(\be)$ is invariant under transposition and $\bDv$ is
		symmetric, so the third term is unchanged when \eqref{E} is transposed,
		while $(\be\bWv-\bWv\be)^\top=\be^\top\bWv-\bWv\be^\top$ by the
		skew-symmetry of $\bWv$. Taking the transpose on each side of \eqref{E}
		therefore yields
		\begin{equation}\label{Et}
			\DDt \be ^\top+\be ^\top\, \bWv- \bWv\, \be ^\top + \bigl(\tr(\be ) - 1\bigr) \bDv\, = -\bg^\top.
		\end{equation}
		Subtracting \eqref{E} from \eqref{Et} and using $\bg^\top=\bg$, we obtain
		\begin{equation*}
			\DDt(\be ^\top-\be )+(\be ^\top-\be )\bWv-\bWv (\be ^\top-\be )= \bm{0}.
		\end{equation*}
		Setting $\bA:=\be^\top-\be$, the symmetry of $\be_0$ gives
		$\bA(x,0)=\bA_0=\bm 0$ in $\om$, and hence
		\begin{equation}\label{od}
			\left\{
			\begin{aligned}
				\p_t\bA+(\bv\cdot\nabla)\bA&=\bWv\bA-\bA \bWv,\quad   &&\text{in}\, \som ,\\
				\bA_0&=\bm 0,\quad &&\text{in}\,\om.
			\end{aligned}
			\right.
		\end{equation}
		Testing \eqref{od} by $\bA$ in the sense of the Frobenius inner product
		$\bA:\bB=\tr(\bA^\top\bB)$ over $\Omega$ yields
		\begin{align*}
			\langle\partial_t\bA,\bA\rangle+\ino(\bv\cdot\nabla)\bA:\bA\,dx=\ino(\bWv\bA-\bA\bWv):\bA\,dx .
		\end{align*}
		By the chain rule, and by integration by parts together with
		\eqref{model+}$_{4,6}$,
		\begin{equation}\label{eq_id1}
			\langle\partial_t\bA,\bA\rangle
			=\frac{1}{2}\frac{d}{dt}\|\bA(t)\|_{L^2(\Omega)}^2,
			\qquad
			\int_{\Omega} (\bv\cdot\nabla)\bA:\bA\,dx
			=-\frac{1}{2}\int_{\Omega}(\nabla\cdot\bv)\,|\bA|^2\,dx=0 .
		\end{equation}
		For the right-hand side, the cyclic property of the trace gives
		\[
		\ino(\bWv\bA-\bA\bWv):\bA\,dx
		=\ino\Bigl(\tr\bigl(\bA^\top\bA\,\bWv\bigr)-\tr\bigl(\bA\bA^\top\bWv\bigr)\Bigr)dx ,
		\]
		and $\bA$ is skew-symmetric by construction, so
		$\bA^\top\bA=\bA\bA^\top=-\bA^2$ and the two traces cancel pointwise.
		Collecting the above identities, we conclude that
		$$\frac{1}{2}\frac{d}{dt}\|\bA(t)\|_{L^2(\om)}^2=0 .$$
		Therefore $\|\bA(t)\|_{L^2(\om)}^2=\|\bA_0\|_{L^2(\om)}^2$, and since
		$\bA_0=\bm 0$,
		$$\|\bA(t)\|_{L^2(\om)}^2=0 \quad \text{for all } t\in[0,T]\quad\implies\quad\bA(x,t)=\bm 0 \quad \text{a.e.\ in }  \som .$$
		By regularity this equality holds pointwise, and hence $\be$ is symmetric.
	\end{proof}
	
	From now on we prescribe the affine remodeling law \eqref{eq:intro-G}. In the
	absence of deformation it describes relaxation toward the isotropic
	homeostatic strain $(\beta/\alpha)\bI$ at rate $\alpha$.
	
	The law \eqref{eq:intro-G} does not satisfy the hypothesis of
	Proposition \ref{symmetry_E_general} a priori, since
	$\bg^\top-\bg=\alpha(\be^\top-\be)$ vanishes precisely when the conclusion
	holds. The argument nevertheless applies.
	\begin{lemma}\label{symmetry_E}
		Let $(\bv,\be)$ be a sufficiently regular solution of \eqref{model+} with
		$\bg$ given by \eqref{eq:intro-G}. Then $\bA:=\be^\top-\be$ satisfies
		\[
		\|\bA(t)\|_{L^2(\om)}= e^{-\alpha t}\,\|\bA(0)\|_{L^2(\om)},
		\qquad t\in[0,T].
		\]
		In particular, if $\be_0=\be_0^\top$ pointwise then
		$\be(x,t)=[\be(x,t)]^\top$ for all $(x,t)\in\som$.
	\end{lemma}
	\begin{proof}
		For $\bg=\alpha\be-\beta\bI$ we have
		$\bg^\top-\bg=\alpha(\be^\top-\be)=\alpha\bA$, since $\bI$ is symmetric. Repeating the subtraction in the proof of
		Proposition \ref{symmetry_E_general} without the symmetry hypothesis on
		$\bg$ therefore gives, in place of \eqref{od},
		\[
		\p_t\bA+(\bv\cdot\nabla)\bA+\alpha\bA=\bWv\bA-\bA\bWv
		\qquad\text{in }\som .
		\]
		Testing by $\bA$ and using \eqref{eq_id1} together with the pointwise
		cancellation of the commutator term, we obtain
		\[
		\frac{1}{2}\frac{d}{dt}\|\bA(t)\|_{L^2(\om)}^{2}+\alpha\|\bA(t)\|_{L^2(\om)}^{2}=0 ,
		\]
		whence $\|\bA(t)\|_{L^2(\om)}=e^{-\alpha t}\|\bA(0)\|_{L^2(\om)}$ for
		$t\in[0,T]$. If in addition $\be_0=\be_0^\top$, then $\bA(0)=\bm 0$ and
		$\be$ is symmetric on $\som$, as in
		Proposition \ref{symmetry_E_general}.
	\end{proof}
	
	\subsection{Decomposition and vanishing-diffusion regularization}
	\label{SubsecDecomposition}
	
	The two nonlinearities in the strain equation are of different types, and
	separating them is the first step of the analysis. Since $\be$ is symmetric by
	Lemma \ref{symmetry_E}, it splits as in \eqref{eq:intro-split} into its
	trace and deviatoric parts:
	\begin{equation}\label{eq:Dec}
		\be = \bF + \frac{q}{3}\bI,
		\qquad
		q = \tr(\be),
		\qquad
		\bF = \dev\be = \be - \frac{q}{3}\bI,
		\qquad
		\tr(\bF) = 0.
	\end{equation}
	Applying \eqref{eq:Dec} to the
	strain equation separates it into two equations of quite different character.
	Taking the trace annihilates the Jaumann commutator, whose trace vanishes by
	the cyclic property, and removes the stretching term as well, since
	$\tr\bDv=\dv\bv=0$; what survives is the closed scalar equation
	\[
	\partial_t q + \bv\cdot\nabla q + \alpha q = 3\beta .
	\]
	The deviatoric part retains both nonlinearities; since $\bI$ commutes with
	$\bWv$, $\dev(\be\bWv-\bWv\be)=\bF\bWv-\bWv\bF$.
	The system \eqref{model+} thus becomes
	\begin{equation}\label{Dmodel}
		\left\{
		\begin{aligned}
			&\rho \partial_t \bv + \rho (\bv \cdot \nabla) \bv + \nabla \pi = \dv(\mu_1 \bD(\bv) + \lambda \bF) + \bff && \quad\text{in } \som ,\\[0.5em]
			&\partial_t q  + \bv \cdot \nabla q + \alpha q = 3\beta && \quad\text{in } \som ,\\[0.5em]
			&\partial_t \bF  + (\bv \cdot \nabla) \bF + \bF \bW(\bv) - \bW(\bv) \bF + (q - 1) \bD(\bv) + \alpha \bF = 0 && \quad\text{in } \som ,\\[0.5em] 
			&\dv \bv = 0 && \quad\text{in } \som ,\\[0.5em]
			&\bv(0) = \bv_0,\,q(0)=q_0 = \tr(\be_0),\, \bF(0) =\bF_0=\dev\,\be_0 && \quad\text{in }\om,\\[0.5em]
			&\bv=0 && \quad\text{on } S\times\partial\om,\\
		\end{aligned}
		\right.
		\tag{\text{{\bf DM}}}
	\end{equation}
	where the isotropic part of the elastic stress has been absorbed into the
	modified pressure
	\[
	\pi := p - \Bigl( \frac{\lambda}{3} + 2\mu \Bigr) q .
	\]
	
	\begin{proposition}[Decomposed equations]\label{thm:decomposed}
		Let \((\bv, \be)\) solve \eqref{model+}. Then \((\bv, q, \bF)\) solves \eqref{Dmodel}. The converse also holds, with $(\bv,\be)$ defined from $(\bv,q,\bF)$ using \eqref{eq:Dec}.
	\end{proposition}
	
	The equation for $q$ is now a linear transport--relaxation equation, decoupled from $\bF$, and its
	solutions inherit pointwise bounds from the initial data.
	These are the bounds under
	which the coupling term $(q-1)\bDv$ is controlled, and we prove them in
	Lemma \ref{lem:qbdd} by a truncation argument, which uses only the energy regularity of the flow. The scalar equation is also amenable to the DiPerna--Lions commutator
	technique \cite{DiPernaLions1989}; see Theorem~\ref{thm:qstrong}.
	
	The decomposition does not improve the regularity of either
	strain variable: the energy estimates still give only
	$q \in L^\infty(0,T;L^2(\om))$ and
	$\bF\in L^\infty(0,T;L^2(\om;\bbS_0))$. We therefore add parabolic terms to
	both strain equations, obtaining the regularized system
	\begin{equation}\label{Rmodel+}
		\left\{
		\begin{aligned}
			&\rho \partial_t \bv^\e + \rho (\bv^\e \cdot \nabla) \bv^\e + \nabla \pi^\e = \dv(\mu_1 \bD(\bv^\e) + \lambda \bF^\e) + \bff && \quad\text{in } \som ,\\[0.5em]
			&\partial_t q^\e - \e \Delta q^\e + \bv^\e \cdot \nabla q^\e + \alpha q^\e = 3\beta && \quad\text{in } \som ,\\[0.5em]
			&\partial_t \bF^\e - \e \Delta \bF^\e + (\bv^\e \cdot \nabla) \bF^\e + \bF^\e \bW(\bv^\e) - \bW(\bv^\e) \bF^\e + (q^\e - 1) \bD(\bv^\e) + \alpha \bF^\e = 0 && \quad\text{in } \som ,\\[0.5em] 
			&\dv \bv^\e = 0 && \quad\text{in } \som ,\\[0.5em]
			&\bv^\e(0) = \bv_0,\,q^\e(0)=q_0 = \tr(\be_0),\, \bF^\e(0) =\bF_0=\dev\,\be_0 && \quad\text{in }\om,\\[0.5em] 
			&\bv^\e=0,\, \e \nabla q^\e\cdot\bm{n}=0,\, \e \nabla\bF^\e\cdot\bm{n}=0 && \quad\text{on } S\times\partial\om,\\
		\end{aligned}
		\right.
		\tag{\text{{\bf RM}}}
	\end{equation}
	for which weak solutions can be constructed by a Galerkin scheme and the
	Aubin--Lions lemma.
	
	\begin{remark}[The regularization as a weakly nonlocal remodeling law]
		\label{rem:Geps}
		The passage from \eqref{Dmodel} to \eqref{Rmodel+} is a change in the
		remodeling law of the original model, as announced in Section
		\ref{sec:intro}. Replacing \eqref{eq:intro-G} by
		\begin{equation}\label{eq:Geps}
			\bg_\e := \alpha\be - \beta\bI - \e\,\Delta\be ,
		\end{equation}
		and imposing $\e\nabla\be\cdot\bm{n}=0$ on $S\times\partial\om$, the strain
		equation of \eqref{model+} reads
		\[
		\partial_t\be - \e\Delta\be + (\bv\cdot\nabla)\be + \be\bWv - \bWv\be
		+ \bigl(\tr\be - 1\bigr)\bDv + \alpha\be - \beta\bI = 0 .
		\]
		Since $\tr(\Delta\be)=\Delta q$ and $\operatorname{dev}(\Delta\be)=\Delta\bF$,
		the Laplacian passes through the decomposition and produces
		\eqref{Rmodel+}$_2$ and \eqref{Rmodel+}$_3$, while
		$\e\nabla\be\cdot\bm{n}=0$ gives the two Neumann conditions in
		\eqref{Rmodel+}$_6$. Moreover
		$\bg_\e^\top-\bg_\e=\alpha(\be^\top-\be)-\e\Delta(\be^\top-\be)$, so the strain
		remains symmetric for every $\e>0$: the proof of Lemma~\ref{symmetry_E} gives
		$\|\bA(t)\|_{L^2(\om)}\le e^{-\alpha t}\|\bA(0)\|_{L^2(\om)}$, with the
		additional nonnegative term $\e\|\nabla\bA\|^2_{L^2(\om)}$ on the left; the
		boundary term vanishes by the Neumann condition.
		
		The law \eqref{eq:Geps} is variational. The affine law \eqref{eq:intro-G} is
		the derivative $\bg=\partial_{\be}\Psi$ of the quadratic remodeling potential
		$\Psi(\be)=\frac{\alpha}{2}|\be|^2-\beta\,\tr\be$, minimized by the
		homeostatic strain $(\beta/\alpha)\bI$; since $\bg$ is a rate, $\Psi$ is a
		rate potential. Adding a penalty on spatial variation of
		the relaxed configuration gives
		\[
		\Psi_\e(\be,\nabla\be)=\frac{\alpha}{2}|\be|^2-\beta\,\tr\be
		+\frac{\e}{2}|\nabla\be|^2 ,
		\]
		whose variational derivative
		$\partial_{\be}\Psi_\e-\operatorname{div}\bigl(\partial_{\nabla\be}\Psi_\e\bigr)$
		is \eqref{eq:Geps}, with the natural boundary condition
		$\e\nabla\be\cdot\bm{n}=0$, since nothing on $\partial\om$ is penalized; for
		conditions of this type derived from a microforce balance we refer to Gurtin
		\cite{Gurtin1996}. Mechanically, remodeling then drives the strain both
		toward the homeostatic state and toward spatial uniformity, so the relaxed
		configurations of neighboring material points are coupled. Since
		$[\alpha]=(\text{time})^{-1}$ and $[\e]=(\text{length})^2(\text{time})^{-1}$,
		this coupling has the intrinsic length $\ell:=\sqrt{\e/\alpha}$, and
		\eqref{eq:Geps} is weakly nonlocal in the sense of \cite{Maugin1979}.
		
		For fixed $\e>0$, \eqref{Rmodel+} is thus itself a model with weakly
		nonlocal remodeling, and the limit $\e\to0$ studied in
		Section~\ref{sec:sec04} is the idealization in which $\ell$ is small compared
		with the length on which the strain varies. The law \eqref{eq:Geps} is an
		interpretation of the regularization; we do not derive it from a
		micromechanical model or identify $\ell$ with a specific tissue length scale.
		Penalizing the volumetric and deviatoric parts separately would give distinct
		coefficients $\e_1,\e_2$ in \eqref{Rmodel+}$_2$ and \eqref{Rmodel+}$_3$; since
		$q^\e$ and $\bF^\e$ are estimated separately, the analysis is unchanged
		provided both tend to zero.
	\end{remark}
	
	\section{Existence of weak solutions for the regularized model \texorpdfstring{\eqref{Rmodel+}}{(RM)}}\label{sec-existence}
	
	We assume the following:
	\begin{assumption}
		\textbf{H1.}
		\assumptionlabel{H1}{H1}
		Let $\be_0 \in L^2(\om;\bbS)$ be such that
		$$0<c_0\leq \tr(\be_0)\leq c_1\quad \text{in $\om$}.$$ 
	\end{assumption}
	\begin{assumption}
		\textbf{H2.}
		\assumptionlabel{H2}{H2}
		There exist $\alpha>0$ and $\beta\geq 0$ such that
		$$\bg=\alpha\be-\beta\bi\quad \text{in $ \som $}.$$
	\end{assumption}
	\begin{assumption}
		\textbf{H3.}
		\assumptionlabel{H3}{H3}
		Let $\f\in L^2(S;H^{-1}(\O))^3$ and  $\bv_0 \in L^2(\om)^3$ with $\div\, \bv_0=0$ in the sense of distributions.
	\end{assumption}

	Under \ref{H1}--\ref{H3}, the admissible space is
	$$\Uc:= \Vc\times\Qc\times\Ec,$$
	where the sets of admissible velocity and strain tensors are given by 
	\begin{align*}
		\Vc
		&:=L^\infty(S;L^2(\O)^3)
		\cap L^2(S;H^1_{0,\div}(\O)^3)
		\cap W^{1,4/3}\bigl(S;(H^1_{0,\div}(\O)^3)^\ast\bigr),\\
		\Qc
		&:=L^\infty(S;L^2(\O))
		\cap L^2(S;H^1(\O))
		\cap L^\infty(S\times\O)
		\cap W^{1,2}\bigl(S;(H^1(\O))^\ast\bigr),\\
		\Ec
		&:=L^\infty(S;L^2(\O;\bbS_0))
		\cap L^2(S;H^1(\O;\bbS_0))
		\cap W^{1,4/3}\bigl(S;(H^1(\O;\bbS_0))^\ast\bigr).
	\end{align*}
	\begin{definition}[Weak solution]\label{Def_Weak}
		For a fixed $\e>0$, the triplet $(\bv^\e,q^\e,\bF^\e)\in\Uc$ is called a weak solution of \eqref{Rmodel+} if it satisfies the initial conditions
		$$\bv^\e(0) = \bv_0, \quad q^\e(0) = \tr(\be_0), \quad \bF^\e(0) = \dev \be_0\quad\text{in }\om$$
		and, for all $\vp\in L^4(S;H^1_{0,\div}(\O)^3)$,
		$\phi\in L^2(S;H^1(\O))$, $\bPS\in L^4(S;H^1(\O;\bbS_0))$, the following hold: 
		\begin{equation}\label{dweak1}
			\rho \ins \langle \partial_t \bv^\e, \vp \rangle - \rho  \inso (\bv^\e  \otimes \bv^\e) : \nabla \vp  + \mu_1  \inso \bD(\bv^\e) : \bD(\vp)
			+ \lambda  \inso \bF^\e : \bD(\vp)  =  \ins\langle \bff, \vp \rangle ,
		\end{equation}
		\begin{equation}\label{dweak2}
			\ins\langle \partial_t q^\e, \phi \rangle + \e  \inso \nabla q^\e \cdot \nabla \phi  -  \inso q^\e \bv^\e \cdot \nabla \phi
			+ \alpha  \inso q^\e \phi = 3\beta\inso \phi ,
		\end{equation}
		
		\begin{multline}\label{dweak3}
			\ins\langle \partial_t \bF^\e, \bPS \rangle + \e  \inso \nabla \bF^\e : \nabla \bPS  -  \inso \bF^\e : (\bv^\e \cdot \nabla) \bPS  +  \inso (q^\e - 1) \bD(\bv^\e) : \bPS \\
			+ \inso (\bF^\e \bW(\bv^\e) - \bW(\bv^\e) \bF^\e) : \bPS + \alpha  \inso \bF^\e : \bPS  = 0.
		\end{multline}
	\end{definition}
	By \eqref{Rmodel+}$_4$ and $\bv^\e=\bm 0$ on $\partial\om$, the transport terms produce no boundary terms.
	
	\begin{remark}[Embeddings]\label{lem:embed}
		For a bounded Lipschitz domain $\O\subset\R^3$,
		\begin{equation}\label{embed_1}
			H^1(\O)\hookrightarrow L^p(\O)\ (1\le p\le6),\qquad
			H^1(\O)\hookrightarrow\hookrightarrow L^p(\O)\ (1\le p<6),
		\end{equation}
		and the same holds for $H^1_0(\O)$. Moreover,
		\begin{equation}\label{embed_2}
			H^2(\O)\hookrightarrow L^\infty(\O).
		\end{equation}
		Since $\bbS_0\cong\R^5$, these embeddings hold componentwise for
		$\bbS_0$-valued functions.
	\end{remark}

	\begin{theorem}
		\label{Thm_existence}
		There exists a weak solution $(\bv^\e,q^\e,\bF^\e)\in\Uc$ of \eqref{Rmodel+}
		in the sense of Definition~\ref{Def_Weak}. Moreover, this solution satisfies
		the energy relations
		\begin{align}
			\label{eq:EEV}
			\frac{\rho}{2}\|\bv^\e(t)\|^2_{L^2(\O)}
			+ \mu_1\|\bD(\bv^\e)\|^2_{L^2(Q_t)}
			+\lambda \int_{Q_t} \bF^\e:\bD(\bv^\e)\,dx\,ds
			\leq\frac{\rho}{2}\|\bv_0\|^2_{L^2(\O)}
			+ \int_0^t\langle\bff,\bv^\e\rangle\,ds,\\
			\label{eq:EEQ}
			\frac{1}{2}\|q^\e(t)\|^2_{L^2(\O)}
			+\e\|\nabla q^\e\|^2_{L^2(Q_t)}
			+\alpha\|q^\e\|^2_{L^2(Q_t)}
			-3\beta \int_{Q_t} q^\e\,dx\,ds
			=\frac{1}{2}\|q_0\|^2_{L^2(\O)},\\
			\label{eq:EEF}
			\frac{1}{2}\|\bF^\e(t)\|^2_{L^2(\O)}
			+\alpha\|\bF^\e\|^2_{L^2(Q_t)}
			+\e\|\nabla\bF^\e\|^2_{L^2(Q_t)}
			+\int_{Q_t}(q^\e-1)\,\bD(\bv^\e):\bF^\e\,dx\,ds
			\leq\frac{1}{2}\|\bF_0\|^2_{L^2(\O)},
		\end{align}
		where \eqref{eq:EEQ} holds for every $t\in[0,T]$ and \eqref{eq:EEV},
		\eqref{eq:EEF} hold for a.e.\ $t\in(0,T)$.
	\end{theorem}
	
	\begin{proof} 
		We prove the existence by a Galerkin procedure for a fixed $\e>0$, organized into ten steps.
		
		\textbf{Step 1: Galerkin basis.}
		Since \(H^1_0(\O)^3\), \(H^1(\O)\), and \(H^1(\O;\bbS_0)\) are separable Hilbert spaces and
		\[
		H^1_0(\O)^3\hookrightarrow\hookrightarrow L^2(\O)^3,\qquad
		H^1(\O)\hookrightarrow\hookrightarrow L^2(\O),
		\qquad
		H^1(\O;\bbS_0)\hookrightarrow\hookrightarrow L^2(\O;\bbS_0),
		\]
		are compact embeddings, and since \(H^1_{0,\div}(\O)^3\) is a closed subspace of \(H^1_0(\O)^3\), there exist countable families
		\[
		\{\vp_i\}_{i\in\mathbb N}\subset H^1_{0,\div}(\O)^3,\qquad
		\{\phi_i\}_{i\in\mathbb N}\subset H^1(\O),
		\qquad
		\{\bPS_i\}_{i\in\mathbb N}\subset H^1(\O;\bbS_0),
		\]
		which are orthonormal and complete in the corresponding \(L^2\)-spaces. Moreover, these families may be chosen so that their finite-dimensional spans are dense in the corresponding energy spaces. 
		
		For further details on the construction of divergence-free bases and Galerkin approximations for incompressible flows, we refer to Temam~\cite{temam1984} and Ladyzhenskaya et al.~\cite{ladyzhenskaya1968}.
		
		\textbf{Step 1.1: Truncated Galerkin approximation.}
		We truncate the scalar variable where it acts as a coefficient. This is needed because, at the Galerkin level, an \(L^\infty\)-bound for \(q_n^\e\) is not yet available, and hence the coefficient \(q_n^\e-1\) in the coupling term cannot be controlled uniformly with respect to \(n\).
		
		Let
		\[
		c_2:=\begin{cases}
			\min\Big\{c_0,\dfrac{3\beta}{\alpha}\Big\}, & \beta>0,\\[2mm]
			c_0e^{-\alpha T}, & \beta=0,
		\end{cases}
		\qquad
		c_3:=\max\Big\{c_1,\frac{3\beta}{\alpha}\Big\},
		\]
		and define the truncation map \(\mathscr T:\R\to\R\) by
		\[
		\mathscr T(s):=\min\{c_3,\max\{c_2,s\}\}.
		\]
		Equivalently,
		\begin{equation}\label{trunc_def}
			\mathscr T(s)=
			\begin{cases}
				c_2, & s<c_2,\\
				s, & c_2\le s\le c_3,\\
				c_3, & s>c_3.
			\end{cases}
		\end{equation}
		Then
		\begin{equation}\label{trunc_bdd}
			|\mathscr T(s)|\le c_3,\qquad
			|\mathscr T(s)-\mathscr T(r)|\le |s-r|,
			\qquad
			|\mathscr T(s)-1|\le c_3+1,
			\qquad r,s\in\R,
		\end{equation}
		and
		\begin{equation}\label{trunc_id}
			\mathscr T(s)=s
			\qquad\text{for all }s\in[c_2,c_3].
		\end{equation}
		The truncation is applied only to the coefficients in both the scalar and tensor equations. After passing to the Galerkin limit and proving the maximum principle for the limiting scalar equation, we shall have \(\mathscr T(q^\e)=q^\e\) a.e. in \(S\times\O\), and hence recover the original regularized system.
		
		For \(n\in\mathbb N\), we look for approximate solutions of the form
		\[
		\bvne(x,t)=\sum_{i=1}^n a_i^n(t)\vp_i(x),\qquad
		\qne(x,t)=\sum_{i=1}^n d_i^n(t)\phi_i(x),\qquad
		\bFne(x,t)=\sum_{i=1}^n b_i^n(t)\bPS_i(x),
		\]
		where the real-valued coefficients $a_i^n$, $d_i^n$, $b_i^n$ are determined by the following finite-dimensional system: for every
		\[
		\vp\in\operatorname{span}\{\vp_1,\ldots,\vp_n\},\qquad
		\phi\in\operatorname{span}\{\phi_1,\ldots,\phi_n\},\qquad
		\bPS\in\operatorname{span}\{\bPS_1,\ldots,\bPS_n\},
		\]
		and for a.e. \(t\in S\),
		\begin{align}
			&\rho\langle \partial_t\bvne,\vp\rangle
			-\rho\ino(\bvne \otimes\bvne):\nabla\vp
			+\mu_1\ino\bD(\bvne):\bD(\vp)
			+\lambda\ino\bFne:\bD(\vp)
			=\langle\bff,\vp\rangle,
			\label{dgweak1}\\
			&\langle\partial_t\qne,\phi\rangle
			+\e\ino\nabla\qne\cdot\nabla\phi
			{ -\ino \mathscr T(\qne)\bvne\cdot\nabla\phi}
			+\alpha\ino\qne\phi
			=3\beta\ino\phi,
			\label{dgweak2}\\
			&\langle\partial_t\bFne,\bPS\rangle
			+\e\ino\nabla\bFne:\nabla\bPS
			-\ino\bFne:(\bvne\cdot\nabla)\bPS
			+\ino(\mathscr T(\qne)-1)\bD(\bvne):\bPS
			\notag\\
			&\hspace{5cm}
			+\ino\bigl(\bFne\bW(\bvne)-\bW(\bvne)\bFne\bigr):\bPS
			+\alpha\ino\bFne:\bPS
			=0.
			\label{dgweak3}
		\end{align}
		The initial data are taken as the corresponding \(L^2\)-orthogonal projections:
		\[
		\bvne(0)=\xi_n\bv_0,\qquad
		\qne(0)=\chi_n q_0,\qquad
		\bFne(0)=\zeta_n\bF_0,
		\]
		where \(\xi_n,\chi_n,\zeta_n\) denote the \(L^2\)-orthogonal projections onto
		\[
		\operatorname{span}\{\vp_1,\ldots,\vp_n\},\qquad
		\operatorname{span}\{\phi_1,\ldots,\phi_n\},\qquad
		\operatorname{span}\{\bPS_1,\ldots,\bPS_n\},
		\]
		respectively.

		\textbf{Step 2: Galerkin ODE system.}
		Writing
		\[
		\ba^n(t):=(a_1^n(t),\ldots,a_n^n(t)),\qquad
		\bd^n(t):=(d_1^n(t),\ldots,d_n^n(t)),\qquad
		\bb^n(t):=(b_1^n(t),\ldots,b_n^n(t)),
		\]
		the system \eqref{dgweak1}--\eqref{dgweak3} is equivalent to a finite-dimensional ODE system
		\begin{equation}\label{d_ode}
			\frac{d}{dt}\bX(t)=\Theta(t,\bX(t)),
			\qquad
			\bX(t):=(\ba^n(t),\bd^n(t),\bb^n(t))\in\R^{3n},
		\end{equation}
		with initial condition
		\[
		\bX(0)=\bigl(\ba^n(0),\bd^n(0),\bb^n(0)\bigr).
		\]
		Here the vector field \(\Theta\) is obtained from the finite-dimensional forms appearing in \eqref{dgweak1}--\eqref{dgweak3}. The terms coming from the Stokes operator, the diffusion operators, the damping terms, and the linear coupling with \(\bFne\) are linear in \(\bX\). The convective terms and the Jaumann term are polynomial in the coefficients. The truncated coupling term
		\[
		(\bd^n,\ba^n)\mapsto
		\int_\O
		\mathscr T\left(\sum_{i=1}^n d_i^n\phi_i\right)
		\bD\left(\sum_{j=1}^n a_j^n\vp_j\right):\bPS_k\,dx
		\]
		is locally Lipschitz with respect to \((\bd^n,\ba^n)\), because \(\mathscr T\) is globally Lipschitz by \eqref{trunc_bdd}. Moreover,
		\[
		t\mapsto \bigl(\langle\bff(t),\vp_1\rangle,\ldots,\langle\bff(t),\vp_n\rangle\bigr)
		\]
		belongs to \(L^2(0,T;\R^n)\subset L^1(0,T;\R^n)\). Hence \(t\mapsto\Theta(t,\bX)\) is measurable for each fixed \(\bX\), and \(\bX\mapsto\Theta(t,\bX)\) is locally Lipschitz for a.e. \(t\in S\). Therefore, by the Carathéodory existence theorem \cite[Chapter~1]{coddingtonLevinson1955}, there exists a maximal time \(T_n^\ast\in(0,T]\) and a solution
		\[
		\bX\in W^{1,1}_{\mathrm{loc}}([0,T_n^\ast);\R^{3n})
		\]
		of \eqref{d_ode}. Consequently, $\bvne,\qne,\bFne$ are absolutely continuous in time with values in the corresponding finite-dimensional spaces.
		
		\textbf{Step 3: Energy estimate and extension to \((0,T)\).}
		We test \eqref{dgweak1} by \(\vp=\bvne\). Since \(\dv\bvne=0\) and \(\bvne=0\) on \(\partial\O\), the convective term vanishes:
		\[
		\ino(\bvne \otimes\bvne):\nabla\bvne=0.
		\]
		Thus
		\begin{equation}\label{En_v}
			\frac{\rho}{2}\frac{d}{dt}\|\bvne\|_{L^2(\O)}^2
			+\mu_1\|\bD(\bvne)\|_{L^2(\O)}^2
			+\lambda\ino\bFne:\bD(\bvne)
			=
			\langle\bff,\bvne\rangle .
		\end{equation}
		Since the truncation map $\mathscr{T}$ is Lipschitz continuous, its primitive
		$\Phi(s):=\int_0^s\mathscr{T}(\tau)\,d\tau$ satisfies
		$\Phi\in C^{1,1}(\mathbb{R})$ and $\Phi'=\mathscr{T}$.
		By the chain rule for Sobolev functions,
		$$\nabla\Phi(q_n^\varepsilon)
		=
		\mathscr{T}(q_n^\varepsilon)\nabla q_n^\varepsilon
		\qquad\text{a.e. in }\Omega.$$
		Since $\nabla\cdot\boldsymbol{v}_n^\varepsilon=0$ in $\Omega$ and
		$\boldsymbol{v}_n^\varepsilon=\bm{0}$ on $\partial\O$, we conclude that
		$$\int_{\Omega}\mathscr{T}(q_n^\varepsilon)\,
		\boldsymbol{v}_n^\varepsilon\cdot\nabla q_n^\varepsilon\,dx=
		\int_{\Omega}
		\boldsymbol{v}_n^\varepsilon\cdot
		\nabla\Phi(q_n^\varepsilon)\,dx 
		=-\int_{\Omega}
		(\nabla\cdot\boldsymbol{v}_n^\varepsilon)\,
		\Phi(q_n^\varepsilon)\,dx
		+\int_{\partial\Omega}
		\Phi(q_n^\varepsilon)\,
		(\boldsymbol{v}_n^\varepsilon\cdot\boldsymbol{n})\,dS=0.$$
		Hence, testing \eqref{dgweak2} by \(\phi=\qne\), we obtain
		\begin{equation}\label{En_q}
			\frac12\frac{d}{dt}\|\qne\|_{L^2(\O)}^2
			+\e\|\nabla\qne\|_{L^2(\O)}^2
			+\alpha\|\qne\|_{L^2(\O)}^2
			=
			3\beta\ino\qne .
		\end{equation}
		
		Finally, testing \eqref{dgweak3} by \(\bPS=\bFne\), we use
		\[
		\ino\bFne:(\bvne\cdot\nabla)\bFne=0\quad\text{and}\quad \bigl(\bFne\bW(\bvne)-\bW(\bvne)\bFne\bigr):\bFne=0
		\qquad\text{a.e. in }\O,
		\]
		because \(\bFne\) is symmetric and \(\bW(\bvne)\) is skew-symmetric. Therefore,
		\begin{equation}\label{En_F}
			\frac12\frac{d}{dt}\|\bFne\|_{L^2(\O)}^2
			+\e\|\nabla\bFne\|_{L^2(\O)}^2
			+\alpha\|\bFne\|_{L^2(\O)}^2
			+\ino(\mathscr T(\qne)-1)\bD(\bvne):\bFne
			=0.
		\end{equation}
		
		Adding \eqref{En_v}, \eqref{En_q}, and \eqref{En_F}, we get
		\begin{multline}\label{En_total_1}
			\frac12\frac{d}{dt}
			\left(
			\rho\|\bvne\|_{L^2(\O)}^2
			+\|\qne\|_{L^2(\O)}^2
			+\|\bFne\|_{L^2(\O)}^2
			\right)
			+\mu_1\|\bD(\bvne)\|_{L^2(\O)}^2\\
			+\e\|\nabla\qne\|_{L^2(\O)}^2
			+\alpha\|\qne\|_{L^2(\O)}^2
			+\e\|\nabla\bFne\|_{L^2(\O)}^2
			+\alpha\|\bFne\|_{L^2(\O)}^2\\
			=
			\langle\bff,\bvne\rangle
			+3\beta\ino\qne
			-\lambda\ino\bFne:\bD(\bvne)
			-\ino(\mathscr T(\qne)-1)\bD(\bvne):\bFne .
		\end{multline}
		
		We estimate the terms on the right-hand side. By Korn's inequality, there exists \(C_K>0\) such that
		\[
		\|\bvne\|_{H^1_0(\O)}
		\le C_K\|\bD(\bvne)\|_{L^2(\O)}.
		\]
		Therefore, by Young's inequality,
		\begin{equation}\label{force_est}
			|\langle\bff,\bvne\rangle|
			\le
			C_K\|\bff\|_{H^{-1}(\O)}\|\bD(\bvne)\|_{L^2(\O)}
			\le
			\frac{\mu_1}{4}\|\bD(\bvne)\|_{L^2(\O)}^2
			+
			\frac{C_K^2}{\mu_1}\|\bff\|_{H^{-1}(\O)}^2 .
		\end{equation}
		Moreover,
		\begin{equation}\label{q_source_est}
			3\beta\ino\qne
			\le
			3\beta\,|\O|^{1/2}\|\qne\|_{L^2(\O)}
			\le
			\frac{\alpha}{2}\|\qne\|_{L^2(\O)}^2
			+
			\frac{9\beta^2}{2\alpha}|\O|.
		\end{equation}
		For the coupling terms, set
		\[
		C_{\lambda,q}:=\lambda+c_3+1.
		\]
		Using \eqref{trunc_bdd}, we have
		\begin{multline}
			\left|
			\lambda\ino\bFne:\bD(\bvne)
			+
			\ino(\mathscr T(\qne)-1)\bD(\bvne):\bFne
			\right|\le
			C_{\lambda,q}
			\|\bD(\bvne)\|_{L^2(\O)}
			\|\bFne\|_{L^2(\O)}\\
			\leq
			\frac{\mu_1}{4}\|\bD(\bvne)\|_{L^2(\O)}^2
			+
			\frac{C_{\lambda,q}^2}{\mu_1}
			\|\bFne\|_{L^2(\O)}^2 .
			\label{coupling_est}
		\end{multline}
		Combining \eqref{En_total_1}--\eqref{coupling_est}, we obtain
		\begin{multline}\label{En_total_2}
			\frac12\frac{d}{dt}
			\left(
			\rho\|\bvne\|_{L^2(\O)}^2
			+\|\qne\|_{L^2(\O)}^2
			+\|\bFne\|_{L^2(\O)}^2
			\right)
			+\frac{\mu_1}{2}\|\bD(\bvne)\|_{L^2(\O)}^2
			+\e\|\nabla\qne\|_{L^2(\O)}^2
			+\frac{\alpha}{2}\|\qne\|_{L^2(\O)}^2
			\\
			+\e\|\nabla\bFne\|_{L^2(\O)}^2
			+\alpha\|\bFne\|_{L^2(\O)}^2
			\le
			\frac{C_K^2}{\mu_1}\|\bff\|_{H^{-1}(\O)}^2
			+
			\frac{9\beta^2}{2\alpha}|\O|
			+
			\frac{C_{\lambda,q}^2}{\mu_1}\|\bFne\|_{L^2(\O)}^2 .
		\end{multline}
		Integrating \eqref{En_total_2} over \((0,t)\), with \(0<t<T_n^\ast\), gives
		\begin{multline}\label{En_total_3}
			\frac12
			\left(
			\rho\|\bvne(t)\|_{L^2(\O)}^2
			+\|\qne(t)\|_{L^2(\O)}^2
			+\|\bFne(t)\|_{L^2(\O)}^2
			\right)
			+\frac{\mu_1}{2}\int_0^t\|\bD(\bvne)\|_{L^2(\O)}^2\,ds
			\\
			+\e\int_0^t\|\nabla\qne\|_{L^2(\O)}^2\,ds
			+\frac{\alpha}{2}\int_0^t\|\qne\|_{L^2(\O)}^2\,ds
			+\e\int_0^t\|\nabla\bFne\|_{L^2(\O)}^2\,ds
			+\alpha\int_0^t\|\bFne\|_{L^2(\O)}^2\,ds
			\\
			\le
			\frac12
			\left(
			\rho\|\bvne(0)\|_{L^2(\O)}^2
			+\|\qne(0)\|_{L^2(\O)}^2
			+\|\bFne(0)\|_{L^2(\O)}^2
			\right)
			\\
			+\frac{C_K^2}{\mu_1}\int_0^t\|\bff(s)\|_{H^{-1}(\O)}^2\,ds
			+\frac{9\beta^2}{2\alpha}|\O|\,t
			+\frac{C_{\lambda,q}^2}{\mu_1}\int_0^t\|\bFne(s)\|_{L^2(\O)}^2\,ds .
		\end{multline}
		In particular, if
		\[
		\mathscr E_n^\e(t):=
		\rho\|\bvne(t)\|_{L^2(\O)}^2
		+\|\qne(t)\|_{L^2(\O)}^2
		+\|\bFne(t)\|_{L^2(\O)}^2,
		\]
		then, since \(\|\bFne(t)\|_{L^2(\O)}^2\le \mathscr E_n^\e(t)\), \eqref{En_total_3} yields
		\[
		\mathscr E_n^\e(t)
		\le
		\mathscr E_n^\e(0)
		+\frac{2C_K^2}{\mu_1}\|\bff\|_{L^2(0,T;H^{-1}(\O))}^2
		+\frac{9\beta^2}{\alpha}|\O|T
		+\frac{2C_{\lambda,q}^2}{\mu_1}
		\int_0^t\mathscr E_n^\e(s)\,ds .
		\]
		By Gronwall's inequality,
		\begin{equation}\label{En_total_4}
			\sup_{0\le t<T_n^\ast}\mathscr E_n^\e(t)
			\le
			\left[
			\mathscr E_n^\e(0)
			+\frac{2C_K^2}{\mu_1}\|\bff\|_{L^2(0,T;H^{-1}(\O))}^2
			+\frac{9\beta^2}{\alpha}|\O|T
			\right]
			\exp\left(\frac{2C_{\lambda,q}^2}{\mu_1}T\right).
		\end{equation}
		Since \(\bvne(0)\), \(\qne(0)\), and \(\bFne(0)\) are the corresponding \(L^2\)-orthogonal projections of
		\(\bv_0\), \(q_0\), and \(\bF_0\), their \(L^2\)-norms are bounded by the norms of the initial data. Hence
		\begin{equation}\label{E_uniform_bound}
			\sup_{0\le t<T_n^\ast}
			\left[
			\|\bvne(t)\|_{L^2(\O)}^2
			+\|\qne(t)\|_{L^2(\O)}^2
			+\|\bFne(t)\|_{L^2(\O)}^2
			\right]
			\le C,
		\end{equation}
		where \(C\) is independent of \(n\) and \(T_n^\ast\).
		
		Returning to \eqref{En_total_3} and using \eqref{En_total_4}, we also obtain
		\[
		\int_0^{T_n^\ast}\|\bD(\bvne)\|_{L^2(\O)}^2\,ds
		+\e\int_0^{T_n^\ast}\|\nabla\qne\|_{L^2(\O)}^2\,ds
		+\e\int_0^{T_n^\ast}\|\nabla\bFne\|_{L^2(\O)}^2\,ds
		\le C,
		\]
		where \(C\) is independent of \(n\) and \(T_n^\ast\). For fixed \(\e>0\), this gives
		\[\bvne \ \text{bounded in } L^2(0,T;H_{0,\div}^1(\O)^3),
		\quad
		\qne \ \text{bounded in } L^2(0,T;H^1(\O)),
		\quad
		\bFne \ \text{bounded in } L^2(0,T;H^1(\O;\bbS_0)),
		\]
		with constants possibly depending on \(\e\), but independent of \(n\).
		
		The bound \eqref{E_uniform_bound} prevents finite-time blow-up of the coefficient vector \(\bX(t)\), because the chosen bases are \(L^2\)-orthonormal. Hence the blow-up alternative for the finite-dimensional ODE system cannot occur. Therefore the maximal existence time satisfies
		\[
		T_n^\ast=T,
		\]
		and the Galerkin solution exists on the whole interval \(S=(0,T)\) for every \(n\in\mathbb N\).
		
		\textbf{Step 4: Uniform estimates on $S$.}
		Since $T_n^\ast=T$, \eqref{En_total_3} and \eqref{E_uniform_bound} hold on
		$[0,T]$. Using $\int_0^t\|\bFne\|^2_{L^2(\O)}\,ds\le CT$ and the fact that
		$L^2$-orthogonal projections do not increase the $L^2$-norms of the initial
		data, we obtain
		\begin{multline}\label{eue_1}
			\sup_{0\le t\le T}
			\left(
			\|\bvne(t)\|_{L^2(\O)}^2
			+\|\qne(t)\|_{L^2(\O)}^2
			+\|\bFne(t)\|_{L^2(\O)}^2
			\right)
			+\int_S\|\bD(\bvne)\|_{L^2(\O)}^2\,ds
			\\
			+\e\int_S\|\nabla\qne\|_{L^2(\O)}^2\,ds
			+\int_S\|\qne\|_{L^2(\O)}^2\,ds
			+\e\int_S\|\nabla\bFne\|_{L^2(\O)}^2\,ds
			+\int_S\|\bFne\|_{L^2(\O)}^2\,ds
			\le C .
		\end{multline}
		Here $C>0$ is independent of $n$ and of $\e$: the right-hand side of
		\eqref{En_total_3} involves only $T$, $\rho$, $\mu_1$, $\alpha$, $\beta$,
		$\lambda$, $c_3$, $C_K$, $|\O|$, the initial data and
		$\|\bff\|_{L^2(S;H^{-1}(\O))}$, while $\e$ enters only on the left-hand side.
		
		For fixed $\e>0$ this gives
		\begin{align}\label{dUE_1}
			&\|\bFne\|_{L^\infty(0,T;L^2(\om))}^2 +  \|\bFne\|_{L^2(0,T;H^1(\om))}^2 \leq C,\quad \|\qne\|_{L^\infty(0,T;L^2(\om))}^2+\|\qne\|_{L^2(0,T;H^1(\om))}^2\leq C,\\ \label{dUE_2}
			&\|\bvne\|_{L^\infty(0,T;L^2(\om))}^2 + \|\bvne\|_{L^2(0,T;H^1(\om))}^2 \leq C.
		\end{align}
		The constant in \eqref{dUE_2} is independent of $\e$; the constants in
		\eqref{dUE_1} depend on $\e$ only through the gradient terms.
		
		Now we estimate the time derivatives in suitable dual spaces.
		First, we estimate \(\p_t\bvne\). Let \(\vp\in H^1_{0,\div}(\O)^3\). From \eqref{dgweak1}, we have
		\[
		\rho\,\langle\partial_t\bvne,\vp\rangle
		=
		\rho\ino(\bvne \otimes\bvne):\nabla\vp
		-\mu_1\ino\bD(\bvne):\bD(\vp)
		-\lambda\ino\bFne:\bD(\vp)
		+\langle\bff,\vp\rangle .
		\]
		Using the embedding \(H^1(\O)\hookrightarrow L^6(\O)\), interpolation $\|\bvne\|_{L^4(\O)}
		\le C\|\bvne\|_{L^2(\O)}^{1/4}
		\|\bvne\|_{H^1(\O)}^{3/4},$
		and Korn's inequality, we obtain
		\begin{align*}
			\left|\rho\ino(\bvne \otimes\bvne):\nabla\vp\right|
			&\le
			C\|\bvne\|_{L^4(\O)}^2\|\nabla\vp\|_{L^2(\O)}
			\le
			C\|\bvne\|_{L^2(\O)}^{1/2}
			\|\bvne\|_{H^1(\O)}^{3/2}
			\|\vp\|_{H^1(\O)},\\
			\left|\mu_1\ino\bD(\bvne):\bD(\vp)\right|
			&\le
			C\|\nabla\bvne\|_{L^2(\O)}\|\vp\|_{H^1(\O)},\quad
			\left|\lambda\ino\bFne:\bD(\vp)\right|
			\le
			C\|\bFne\|_{L^2(\O)}\|\vp\|_{H^1(\O)}.
		\end{align*}
		Hence
		\[
		\|\partial_t\bvne\|_{(H^1_{0,\div}(\O)^3)^\ast}
		\le
		C\left(
		\|\bvne\|_{L^2(\O)}^{1/2}\|\bvne\|_{H^1(\O)}^{3/2}
		+\|\nabla\bvne\|_{L^2(\O)}
		+\|\bFne\|_{L^2(\O)}
		+\|\bff\|_{H^{-1}(\O)}
		\right).
		\]
		Using \eqref{dUE_1}--\eqref{dUE_2}, together with the embedding $L^2(0,T)\hookrightarrow L^{4/3}(0,T),$
		it follows that each term on the right-hand side belongs to
		$L^{4/3}(0,T)$. Consequently, we conclude
		\begin{equation}\label{dv_dual}
			\|\partial_t\bvne\|_{L^{4/3}(S;(H^1_{0,\div}(\O)^3)^\ast)}
			\le C .
		\end{equation}
		Here \(C\) is independent of \(n\).
		
		Next, we estimate \(\partial_t\qne\). Let \(\phi\in H^1(\O)\). From \eqref{dgweak2},
		\[
		\langle \partial_t\qne,\phi\rangle
		=
		-\e\ino\nabla\qne\cdot\nabla\phi
		+\ino\mathscr T(\qne)\bvne\cdot\nabla\phi
		-\alpha\ino\qne\phi
		+3\beta\ino\phi .
		\]
		The linear terms are estimated as
		\begin{align*}
			\left|\e\ino\nabla\qne\cdot\nabla\phi\right|
			&\le
			\e\|\nabla\qne\|_{L^2(\O)}\|\phi\|_{H^1(\O)},\quad
			\left|\alpha\ino\qne\phi\right|
			\le
			C\|\qne\|_{L^2(\O)}\|\phi\|_{H^1(\O)}.
		\end{align*}
		For the transport term, using boundedness in \eqref{trunc_bdd}, we deduce
		$$\left|\ino \mathscr{T}( \qne)\bvne\cdot\nabla\phi\right|
		\le
		c_3\|\bvne\|_{L^2(\O)}
		\|\nabla\phi\|_{L^2(\O)}
		\leq c_3 \|\bvne\|_{L^2(\O)}\|\phi\|_{H^1(\O)}.$$
		Therefore,
		$$\|\partial_t\qne\|_{(H^1(\O))^\ast}
		\le
		\e\|\nabla\qne\|_{L^2(\O)}
		+ C \|\bvne\|_{L^2(\O)}
		+C\|\qne\|_{L^2(\O)}
		+C.$$
		Arguing as for \eqref{dv_dual}, \eqref{dUE_1}--\eqref{dUE_2} give
		\begin{equation}\label{dq_dual}
			\|\partial_t\qne\|_{ L^{2}(S;(H^1(\O))^\ast)}
			\le C_\e .
		\end{equation}
		
		Finally, we estimate \(\partial_t\bFne\). Let \(\bPS\in H^1(\O;\bbS_0)\). From \eqref{dgweak3},
		\begin{multline*}
			\langle\partial_t\bFne,\bPS\rangle
			=
			-\e\ino\nabla\bFne:\nabla\bPS
			+\ino\bFne:(\bvne\cdot\nabla)\bPS
			-\ino(\mathscr T(\qne)-1)\bD(\bvne):\bPS
			\\
			-\ino\bigl(\bFne\bW(\bvne)-\bW(\bvne)\bFne\bigr):\bPS
			-\alpha\ino\bFne:\bPS .
		\end{multline*}
		The interpolation inequality, the embedding \(H^1(\O)\hookrightarrow L^6(\O)\) and \eqref{trunc_bdd} yield
		\begin{align*}
			&\begin{aligned}
				\left|\e\ino\nabla\bFne:\nabla\bPS\right|
				&\le
				\e\|\nabla\bFne\|_{L^2(\O)}
				\|\bPS\|_{H^1(\O)}.
			\end{aligned}\\
			&\begin{aligned}
				\left|\ino\bFne:(\bvne\cdot\nabla)\bPS\right|
				\le
				\|\bFne\|_{L^3(\O)}
				\|\bvne\|_{L^6(\O)}
				\|\nabla\bPS\|_{L^2(\O)}
				\le
				C\|\bFne\|_{L^2(\Omega)}^{1/2} \|\bFne\|_{H^1(\Omega)}^{1/2}
				\|\bvne\|_{H^1(\O)}
				\|\bPS\|_{H^1(\O)}.
			\end{aligned}\\
			&\begin{aligned}
				\left|
				\ino(\mathscr T(\qne)-1)\bD(\bvne):\bPS
				\right|
				\le
				(c_3+1)\|\bD(\bvne)\|_{L^2(\O)}
				\|\bPS\|_{L^2(\O)}
				\le
				C\|\nabla\bvne\|_{L^2(\O)}
				\|\bPS\|_{H^1(\O)}.
			\end{aligned}\\
			&\begin{aligned}
				\left|
				\ino\bigl(\bFne\bW(\bvne)-\bW(\bvne)\bFne\bigr):\bPS
				\right|
				\le
				C\|\bFne\|_{L^3(\O)} &\|\nabla\bvne\|_{L^2(\O)}
				\|\bPS\|_{L^6(\O)}\\
				&\le C\|\bFne\|_{L^2(\Omega)}^{1/2} \|\bFne\|_{H^1(\Omega)}^{1/2}
				\|\nabla\bvne\|_{L^2(\O)}
				\|\bPS\|_{H^1(\O)}.
			\end{aligned}
		\end{align*}
		Therefore,
		\begin{multline*}
			\|\partial_t\bFne\|_{(H^1(\O;\bbS_0))^\ast}
			\le
			\e\|\nabla\bFne\|_{L^2(\O)}
			+C\|\bFne\|_{L^2(\Omega)}^{1/2} \|\bFne\|_{H^1(\Omega)}^{1/2}\|\bvne\|_{H^1(\O)}
			+C\|\nabla\bvne\|_{L^2(\O)}
			\\
			+C\|\bFne\|_{L^2(\Omega)}^{1/2} \|\bFne\|_{H^1(\Omega)}^{1/2}\|\nabla\bvne\|_{L^2(\O)}
			+C\|\bFne\|_{L^2(\O)}.
		\end{multline*}
		Consequently, arguing as in the derivation of \eqref{dv_dual}--\eqref{dq_dual} and using the estimates \eqref{dUE_1}--\eqref{dUE_2}, we obtain
		\begin{equation}\label{dF_dual}
			\|\partial_t\bFne\|_{L^{4/3}(S;(H^1(\O;\bbS_0))^\ast)}
			\le C_\e .
		\end{equation}

		\textbf{Step 5: Compactness.}
		Combining \eqref{dUE_1}--\eqref{dUE_2} with \eqref{dv_dual}, \eqref{dq_dual}, and \eqref{dF_dual}, we have
		\begin{align}
			\|\bvne\|_{L^\infty(S;L^2(\O)^3)}
			+\|\bvne\|_{L^2(S;H^1_{0,\div}(\O)^3)}
			+\|\partial_t\bvne\|_{L^{4/3}(S;(H^1_{0,\div}(\O)^3)^\ast)}
			&\le C, \label{dv_bdd}\\
			\|\qne\|_{L^\infty(S;L^2(\O))}
			+\|\qne\|_{L^2(S;H^1(\O))}
			+\|\partial_t\qne\|_{ L^{2}(S;(H^1(\O))^\ast)}
			&\le C_\e, \label{dq_bdd}\\
			\|\bFne\|_{L^\infty(S;L^2(\O;\bbS_0))}
			+\|\bFne\|_{L^2(S;H^1(\O;\bbS_0))}
			+\|\partial_t\bFne\|_{L^{4/3}(S;(H^1(\O;\bbS_0))^\ast)}
			&\le C_\e. \label{dF_bdd}
		\end{align}
		Here \(C\) is independent of \(n\), and \(C_\e\) is independent of \(n\), but may depend on \(\e\).
		
		Hence, up to a subsequence, there exist $\bv^\e, q^\e, \bF^\e$
		such that
		\begin{align*}
			\bvne &\rightharpoonup \bv^\e \quad
			\text{weakly in }L^2(S;H^1_{0,\div}(\O)^3),\qquad
			&\bvne &\overset{\ast}{\rightharpoonup}\bv^\e
			\quad\text{weakly-$\ast$ in }L^\infty(S;L^2(\O)^3),\\
			\qne &\rightharpoonup q^\e \quad
			\text{weakly in }L^2(S;H^1(\O)),\qquad
			&\qne &\overset{\ast}{\rightharpoonup} q^\e
			\quad\text{weakly-$\ast$ in }L^\infty(S;L^2(\O)),\\
			\bFne &\rightharpoonup \bF^\e \quad
			\text{weakly in }L^2(S;H^1(\O;\bbS_0)),\qquad
			&\bFne &\overset{\ast}{\rightharpoonup}\bF^\e
			\quad\text{weakly-$\ast$ in }L^\infty(S;L^2(\O;\bbS_0)).
		\end{align*}
		Moreover, by the Aubin--Lions lemma,
		\begin{align}
			\bvne &\sc \bv^\e\qquad\text{strongly in }L^2(S;L^2(\O)^3), \label{dvstrong}\\
			\qne &\sc q^\e\qquad\text{strongly in }L^2(S;L^2(\O)), \label{dqstrong}\\
			\bFne &\sc \bF^\e\qquad\text{strongly in }L^2(S;L^2(\O;\bbS_0)). \label{dFstrong}
		\end{align}
		In particular, up to a further subsequence, these convergences also hold almost everywhere in \(S\times\O\).
		
		\textbf{Step 5.1: Nonlinear terms.} The strong convergences \eqref{dvstrong}--\eqref{dFstrong} yield
		\begin{align}
			\bvne\otimes\bvne &\sc \bv^\e\otimes\bv^\e
			&& \text{strongly in } L^1(S;L^1(\O)^{3\X3}), \label{dv_L1_strong}\\
			\bFne\otimes\bvne &\sc \bF^\e\otimes\bv^\e
			&& \text{strongly in } L^1(S;L^1(\O)^{3\X3\X3}). \label{dF_L1_strong}
		\end{align}
		Since $\mathscr T$ is Lipschitz and bounded by \eqref{trunc_bdd}, \eqref{dqstrong} gives
		\begin{equation}\label{tqstrong}
			\mathscr T(\qne)\sc \mathscr T(q^\e)
			\quad\text{strongly in }L^2(S\times\O) \quad\text{and}\quad \mathscr{T}(q_n^\varepsilon) \sc \mathscr{T}(q^\varepsilon) \quad \text{a.e. in } S\times\Omega.
		\end{equation}
		Combined with \eqref{dvstrong} and \eqref{dFstrong}, this gives
		\begin{align}
			\mathscr{T}(q_n^\varepsilon)\,\bm{v}_n^\varepsilon &\longrightarrow\; \mathscr{T}(q^\varepsilon)\,\bm{v}^\varepsilon
			\quad &\text{strongly in } L^2(S;L^2(\Omega)^3), \label{dq_L1_strong}\\
			(\mathscr T(\qne)-1)\bFne &\longrightarrow\; (\mathscr{T}(q^\varepsilon)-1)\,\bm{F}^\varepsilon
			\quad &\text{strongly in } L^2(S;L^2(\Omega;\bbS_0)).\label{dT_L1_strong}
		\end{align}
		
		\textbf{Step 6: Limit in the momentum equation.}
		From \eqref{dv_dual}, we have
		\[
		\partial_t\bvne\rightharpoonup\partial_t\bv^\e
		\quad\text{weakly in }L^{4/3}(S;(H^1_{0,\div}(\O)^3)^\ast).
		\]
		Therefore, for every \(\vp\in L^4(S;H^1_{0,\div}(\O)^3)\),
		\[
		\int_S\langle\partial_t\bvne,\vp\rangle\,dt
		\to
		\int_S\langle\partial_t\bv^\e,\vp\rangle\,dt.
		\]
		Since $S=(0,T)$ has finite measure, the linear terms pass to the limit by the weak convergences above:
		\[
		\inso \bD(\bvne):\bD(\vp)
		\to
		\inso \bD(\bv^\e):\bD(\vp),
		\quad\text{and}\quad
		\inso \bFne:\bD(\vp)
		\to
		\inso \bF^\e:\bD(\vp).
		\]
		Since $\bvne\otimes\bvne$ is bounded in the reflexive space $L^{4/3}(S;L^2(\O)^{3\X3})$, \eqref{dv_L1_strong} implies
		$$(\bvne \otimes\bvne) \wc (\bv^\e \otimes\bv^\e)\qquad \text{weakly in } L^{4/3}(S;L^2(\O)^{3\X3}).$$
		Consequently,
		\[
		\inso
		(\bvne \otimes\bvne):\nabla\vp
		\sc
		\inso
		(\bv^\e \otimes\bv^\e):\nabla\vp,
		\]
		for all $\vp\in L^4(S;H^1_{0,\div}(\Omega)^3)$. Hence \(\bv^\e\) satisfies the limiting momentum equation.
		
		\textbf{Step 7: Limit in the scalar equation.} From \eqref{dq_dual}, we have
		$$\partial_t\qne\rightharpoonup\partial_t q^\e
		\quad\text{weakly in }L^{2}(S;(H^1(\O))^\ast).$$
		Thus, for every \(\phi\in L^2(S;H^1(\O))\),
		\[
		\int_S\langle\partial_t\qne,\phi\rangle\,dt
		\sc
		\int_S\langle\partial_t q^\e,\phi\rangle\,dt.
		\]
		The diffusion and damping terms pass to the limit by weak convergence:
		\[
		\e\inso\nabla\qne\cdot\nabla\phi
		\to
		\e\inso\nabla q^\e\cdot\nabla\phi,\quad
		\text{and}\quad
		\alpha\inso\qne\phi
		\to
		\alpha\inso q^\e\phi .
		\]
		Moreover, for all $\phi\in L^2(S;H^1(\O))$, the strong convergence \eqref{dq_L1_strong} yields
		$$\inso \mathscr T(\qne)\bvne\cdot\nabla\phi \sc \inso \mathscr T(q^\e) \bv^\e\cdot\nabla\phi .$$
		Therefore, \(q^\e\) satisfies the limiting scalar equation.
		
		\textbf{Step 8: Limit in the tensor equation.}
		From \eqref{dF_dual}, we have
		\[
		\partial_t\bFne\rightharpoonup\partial_t\bF^\e
		\quad\text{weakly in }L^{4/3}(S;(H^1(\O;\bbS_0))^\ast).
		\]
		Thus, for every \(\bPS\in L^4(S;H^1(\O;\bbS_0))\),
		\[
		\int_S\langle\partial_t\bFne,\bPS\rangle\,dt
		\sc
		\int_S\langle\partial_t\bF^\e,\bPS\rangle\,dt.
		\]
		The diffusion and damping terms pass to the limit by weak convergence:
		\[
		\e\inso\nabla\bFne:\nabla\bPS
		\sc
		\e\inso\nabla\bF^\e:\nabla\bPS,\quad
		\text{and}\quad
		\alpha\inso\bFne:\bPS
		\sc
		\alpha\inso\bF^\e:\bPS .
		\]
		As in Step~6, since $\bFne\otimes\bvne$ is bounded in $L^{4/3}(S;L^2(\O)^{3\X3\X3})$ by
		$\|\bFne\otimes\bvne\|_{L^2(\O)}\le\|\bFne\|_{L^3(\O)}\|\bvne\|_{L^6(\O)}$
		and \eqref{dF_bdd}, \eqref{dF_L1_strong} gives 
		\[
		\inso
		\bFne:(\bvne\cdot\nabla)\bPS
		\sc
		\inso
		\bF^\e:(\bv^\e\cdot\nabla)\bPS ,
		\]
		for all $\bPS\in L^4(S;H^1(\O;\bbS_0))$. Moreover, since $L^4(S;H^1(\Omega;\mathbb{S}_0))$ embeds into $L^2(S;L^2(\Omega;\mathbb{S}_0))$, the convergence~\eqref{tqstrong} together with 
		the uniform bound~\eqref{trunc_bdd} and the dominated convergence theorem, implies that
		\[
		(\mathscr{T}(\qne)-1)\bPS
		\longrightarrow
		(\mathscr{T}(q^\varepsilon)-1)\bPS
		\quad\text{strongly in }L^2(S;L^2(\Omega;\mathbb{S}_0)).
		\]
		Therefore, using the weak convergence of $\bm{D}(\bvne)$, we conclude that for every \(\bPS\in L^4(S;H^1(\O;\bbS_0))\),
		\[
		\int_{S\times\Omega}
		(\mathscr{T}(\qne)-1)\bm{D}(\bvne):\bm{\Psi}\,dx\,dt
		\longrightarrow
		\int_{S\times\Omega}
		(\mathscr{T}(q^\varepsilon)-1)\bm{D}(\bv^\e):\bm{\Psi}\,dx\,dt.
		\]
		Finally, the estimate of the Jaumann term in Step~4 gives
		$$\|\bFne\bW(\bvne)\|_{L^{6/5}(\O)}\le\|\bFne\|_{L^3(\O)}\|\nabla\bvne\|_{L^2(\O)}
		\le C\|\bFne\|^{1/2}_{L^2(\O)}\|\bFne\|^{1/2}_{H^1(\O)}\|\nabla\bvne\|_{L^2(\O)}.$$
		By \eqref{dF_bdd}, the commutator
		$\bJc^\e_n:=\bFne\bW(\bvne)-\bW(\bvne)\bFne$ is therefore bounded in the
		reflexive space $L^{4/3}(S;L^{6/5}(\O)^{3\times3})$, whose dual
		$L^4(S;L^6(\O)^{3\times3})$ contains $L^4(S;H^1(\O;\bbS_0))$ by
		\eqref{embed_1}. By \eqref{dFstrong} and $\bW(\bvne)\wc\bW(\bv^\e)$ weakly in
		$L^2(\som)$, we have
		$\bJc^\e_n\to\bF^\e\bW(\bv^\e)-\bW(\bv^\e)\bF^\e$ in $\mathcal D'(\som)$.
		Hence the weak limit in $L^{4/3}(S;L^{6/5}(\O)^{3\times3})$ is identified, and
		\[
		\inso
		\bigl(\bFne\bW(\bvne)-\bW(\bvne)\bFne\bigr):\bPS
		\sc
		\inso
		\bigl(\bF^\e\bW(\bv^\e)-\bW(\bv^\e)\bF^\e\bigr):\bPS,
		\]
		for every $\bPS\in L^4(S;H^1(\O;\bbS_0)).$
		
		Collecting the above convergences, the limit triple \((\bv^\e,q^\e,\bF^\e)\) satisfies the weak formulation of the truncated regularized system.
		
		\textbf{Step 9: Removal of the truncation.}
		Since $(\bv^\e,q^\e)$ is a weak solution of the truncated system obtained in Step~8, Lemma~\ref{lem:qbdd} yields 
		$$c_2\le q^\e(x,t)\le c_3
		\qquad\text{a.e. } (x,t)\in S\times\O .$$
		Therefore, by \eqref{trunc_id},
		$$\mathscr T(q^\e)=q^\e
		\qquad\text{a.e. in }S\times\O .$$
		Consequently,
		$$\mathscr T(q^\e)\bv^\e=q^\e\bv^\e \quad\text{and}\quad
		(\mathscr T(q^\e)-1)\bD(\bv^\e)=(q^\e-1)\bD(\bv^\e)
		\qquad\text{a.e. in }S\times\O .$$
		Hence, every truncated term in the weak formulation coincides with its corresponding original term. It follows that the weak formulation of the truncated regularized system reduces to the weak formulation of the original regularized system. Therefore, the limit triple \((\bv^\e,q^\e,\bF^\e)\) is a weak solution of the original regularized system \eqref{Rmodel+} in the sense of Definition~\ref{Def_Weak}.
		
		\textbf{Step 10: Energy relations at the \(\varepsilon\)-level.}
		We now derive the energy relations for the weak solution
		\((\bv^\e,q^\e,\bF^\e)\).
		
		Using the bound \eqref{dq_bdd} we have $q^\e\in L^2(S;H^1(\O))$ and $
		\partial_tq^\e\in L^2(S;(H^1(\O))^\ast)$.
		Thus, the Lions--Magenes lemma (see \cite{lionsMagenes1972}) gives
		\[
		q^\e\in C([0,T];L^2(\O)),\quad\text{and}\qquad \frac{1}{2}\frac{d}{dt}\|q^\e(t)\|_{L^2(\O)}^2
		=
		\langle \partial_tq^\e(t),q^\e(t)\rangle
		\quad\text{for a.e. }t\in S.
		\]
		Therefore, taking \(\phi=q^\e\) in the scalar equation \eqref{dweak2}, and arguing as in the derivation of \eqref{En_q} and integrating over \((0,t)\), we obtain the energy identity \eqref{eq:EEQ}, for every \(t\in[0,T]\). Next, we have
		$$\frac{\rho}{2}\|\bm{v}_n^\e(t)\|_{L^2(\Omega)}^2
		+\mu_1\int_0^t\|\bm{D}(\bm{v}_n^\e)\|_{L^2(\Omega)}^2\,ds
		+\lambda\int_{Q_t}\bm{F}_n^\e:\bm{D}(\bm{v}_n^\e)\,dx\,ds =\frac{\rho}{2}\|\bm{v}_{n}^\e (0)\|_{L^2(\Omega)}^2
		+\int_0^t\langle\bm{f},\bm{v}_n^\e\rangle\,ds.$$
		
		Passing to the limit as $n\to\infty$, the weak lower semicontinuity of the $L^2$-norm yields
		$$\|\bm{v}^\e(t)\|_{L^2(\Omega)}^2
		\le
		\liminf_{n\to\infty}
		\|\bm{v}_n^\e(t)\|_{L^2(\Omega)}^2 \quad\text{and}\quad \int_0^t\|\bm{D}(\bm{v}^\e)\|_{L^2(\Omega)}^2\,ds
		\le
		\liminf_{n\to\infty}
		\int_0^t\|\bm{D}(\bm{v}_n^\e)\|_{L^2(\Omega)}^2\,ds.$$
		By \eqref{dvstrong}, for a.e.\ $t$ and along a further subsequence,
		$$\bvne(t)\wc \bv^\e(t)
		\quad\text{weakly in }L^2(\O)^3 .$$
		
		The strong convergence \eqref{dFstrong}, the weak convergence of $\bD(\bvne)$ and the convergence of the initial data and forcing term allow us to pass to the limit in the remaining terms, which gives \eqref{eq:EEV}.
		
		Similarly, for the tensor equation, the Galerkin energy identity is
		\begin{multline*}
			\frac12\|\bFne(t)\|^2_{L^2(\O)}
			+\e\int_0^t\|\nabla\bFne\|^2_{L^2(\O)}\,ds
			+\alpha\int_0^t\|\bFne\|^2_{L^2(\O)}\,ds
			\\
			+\int_{Q_t}(\mathscr T(\qne)-1)\bD(\bvne):\bFne\,dx\,ds
			=
			\frac12\|\bFne(0)\|^2_{L^2(\O)}.
		\end{multline*}
		The transport and Jaumann terms vanish by the same argument as in \eqref{En_F}.
		By \eqref{dFstrong}, for a.e.\ $t$ and along a further subsequence,
		$$\bFne(t)\wc \bF^\e(t)
		\quad\text{weakly in }L^2(\O;\bbS_0).$$
		Passing to the limit as $n\to\infty$, the weak lower semicontinuity of the $L^2$-norm and the gradient norm, together with the strong convergence \eqref{dFstrong}, \eqref{dT_L1_strong} 
		and the weak convergence of $\bm{D}(\bm{v}_n^\e)$, allow us to pass to the limit in the remaining terms. Consequently, removing the truncation yields \eqref{eq:EEF} for a.e.\ $t\in(0,T)$. 
		
		This completes the proof of Theorem~\ref{Thm_existence}.
	\end{proof}
	
	\section{Asymptotic analysis of the \texorpdfstring{$\varepsilon$}{epsilon}-dependent regularized model \texorpdfstring{\eqref{Rmodel+}}{(RM)}}
	\label{sec:sec04}
	Throughout this section we assume $0<\e\le1$.

	\begin{lemma}\label{e_Ues}
		Let $(\bv^\e,q^\e,\bF^\e)$ be the weak solution of \eqref{Rmodel+} constructed in
		Theorem~\ref{Thm_existence}. There is $C>0$, independent of $\e$, such that
		\begin{align}
			\label{U_v_bdd}
			\|\bv^\e\|_{L^\infty(0,T;L^2(\om))}^2 + \|\bv^\e\|_{L^2(0,T;H^1(\om))}^2
			+ \|\partial_t \bv^\e\|_{L^{4/3}(S;\,(H^1_{0,\div}(\om)^3)^*)}
			&\leq C,\\
			\label{U_q_bdd}
			\|q^\e\|_{L^\infty(0,T;L^2(\om))}^2 + \|q^\e\|_{L^2(0,T;L^2(\om))}^2
			+ \e\|\nabla q^\e\|_{L^2(0,T;L^2(\om))}^2
			+ \|\partial_t q^\e\|_{L^2(S;\,(H^1(\O))^\ast)}
			&\leq C,\\
			\label{U_F_bdd}
			\|\bF^\e\|_{L^\infty(0,T;L^2(\om))}^2 + \|\bF^\e\|_{L^2(0,T;L^2(\om))}^2
			+ \e\|\nabla \bF^\e\|_{L^2(0,T;L^2(\om))}^2
			+ \|\partial_t \bF^\e\|_{L^2(S;\,(H^2(\O;\bbS_0))^\ast)}
			&\leq C.
		\end{align}
	\end{lemma}
	
	\begin{proof}
		
		\textbf{Step 1: Uniform estimates.}
		Let $\mathscr E^\e:=\rho\|\bv^\e\|^2_{L^2(\O)}+\|q^\e\|^2_{L^2(\O)}+\|\bF^\e\|^2_{L^2(\O)}$.
		We add \eqref{eq:EEV}--\eqref{eq:EEF}, move the two coupling terms to the
		right-hand side and estimate as in \eqref{force_est}--\eqref{coupling_est},
		using $|q^\e-1|\le c_3+1$ from \eqref{q_bdd} in place of \eqref{trunc_bdd}.
		For a.e.\ $t\in S$ this gives
		\[
		\tfrac12\mathscr E^\e(t)+\tfrac{\mu_1}{2}\|\bD(\bv^\e)\|^2_{L^2(Q_t)}
		+\e\|\nabla q^\e\|^2_{L^2(Q_t)}+\e\|\nabla\bF^\e\|^2_{L^2(Q_t)}
		\le C_0+\frac{C_{\lambda,q}^2}{\mu_1}\int_0^t\mathscr E^\e\,ds,
		\]
		where $C_0:=\tfrac12\mathscr E^\e(0)+\tfrac{C_K^2}{\mu_1}\|\bff\|^2_{L^2(S;H^{-1}(\O))}
		+\tfrac{9\beta^2}{2\alpha}|\O|T$. Neither $C_0$ nor $C_{\lambda,q}=\lambda+c_3+1$
		depends on $\e$. Gronwall's inequality in integral form, applicable since
		$\mathscr E^\e\in L^\infty(S)$, and Korn's inequality then give
		\begin{equation}\label{eUE}
			\begin{aligned}
				\|\bv^\e\|_{L^\infty(S;L^2(\om))}^2 + \|\bv^\e\|_{L^2(S;H^1(\om))}^2 &\leq C,\\
				\|q^\e\|_{L^\infty(S;L^2(\om))}^2 + \|q^\e\|_{L^2(S;L^2(\om))}^2
				+ \e\|\nabla q^\e\|_{L^2(S;L^2(\om))}^2 &\leq C,\\
				\|\bF^\e\|_{L^\infty(S;L^2(\om))}^2 + \|\bF^\e\|_{L^2(S;L^2(\om))}^2
				+ \e\|\nabla \bF^\e\|_{L^2(S;L^2(\om))}^2 &\leq C,
			\end{aligned}
		\end{equation}
		with $C>0$ independent of $\e$.

		\textbf{Step 2: Dual-space estimates.}
		For $\vp\in H^1_{0,\div}(\om)^3$, \eqref{dweak1} gives
		\[
		\rho\,\langle\partial_t\bv^\e,\vp\rangle
		= \rho\int_\om(\bv^\e \otimes\bv^\e):\nabla\vp
		- \mu_1\ino \bD(\bv^\e):\bD(\vp)
		- \lambda\ino\bF^\e:\bD(\vp)
		+ \langle\bff,\vp\rangle.
		\]
		Proceeding exactly as in the derivation of the dual estimate for $\bvne$ in Step~4 of Theorem~\ref{Thm_existence}, we obtain
		\begin{equation}\label{e_v_dual}
			\|\partial_t\bv^\e\|_{L^{4/3}(S;\,(H^1_{0,\div}(\om)^3)^*)} \leq C,
		\end{equation}
		with $C$ independent of $\e$.   
		
		For $\phi\in H^1(\om)$, \eqref{dweak2} gives
		\[
		\langle\partial_t q^\e,\phi\rangle
		= -\e\ino\nabla q^\e\cdot\nabla\phi
		+ \ino q^\e\bv^\e\cdot\nabla\phi
		- \alpha\ino q^\e\phi
		+ 3\beta\ino\phi.
		\]
		Using the $L^\infty$-bound \eqref{q_bdd} on $q^\e$ and the estimates \eqref{eUE}, the Cauchy--Schwarz inequality yields
		\begin{align*}
			&\left|\e\ino\nabla q^\e\cdot \nabla\phi\right|\leq \e\|\nabla q^\e\|_{L^2(\om)} \|\phi\|_{H^1(\om)}=\sqrt{\e}\,\|\sqrt{\e}\,\nabla q^\e\|_{L^2(\om)}\|\phi\|_{H^1(\om)}, \quad
			\left|\alpha \ino q^\e \phi\right|\leq  C\,\|q^\e\|_{L^2(\om)} \|\phi\|_{H^1(\om)}\\
			&\left|\ino q^\e\bv^\e\cdot\nabla\phi\right|\leq  \|q^\e\|_{L^\infty(\Omega)}\|\bv^\e\|_{L^2(\Omega)} \|\nabla \phi\|_{L^2(\Omega)}\leq c_3\|\bv^\e\|_{L^2(\Omega)}\|\phi\|_{H^1(\Omega)}, \quad
			\left|3\beta \ino\phi\right|\leq C\,\|\phi\|_{H^1(\om)}.
		\end{align*}
		Combining the above, we obtain
		$$\sup_{\|\phi\|_{H^1(\om)}\,\leq 1}\,\left|\langle\partial_t q^\e,\phi\rangle\right|\leq C\big(\|\sqrt{\e}\,\nabla q^\e\|_{L^2(\om)}+\|\bv^\e\|_{L^2(\Omega)}+\|q^\e\|_{L^2(\om)}+1\big).$$
		Squaring, integrating over $S$ and using \eqref{eUE}, in particular
		$\|\sqrt{\e}\,\nabla q^\e\|_{L^2(\som)}\le C$, gives
		\begin{equation}\label{e_q_dual}
			\|\partial_t q^\e\|_{L^2(S;(H^1(\O))^\ast)}\leq C.
		\end{equation}
		
		For $\bPS\in H^2(\om;\bbS_0)$,
		\eqref{dweak3} gives
		\begin{multline*}
			\langle\partial_t\bF^\e,\bPS\rangle
			= -\e\ino\nabla\bF^\e:\nabla\bPS
			+ \ino\bF^\e:(\bv^\e\cdot\nabla)\bPS
			- \ino(q^\e-1)\bD(\bv^\e):\bPS\\
			- \ino(\bF^\e\bW(\bv^\e)-\bW(\bv^\e)\bF^\e):\bPS
			- \alpha\ino\bF^\e:\bPS.
		\end{multline*}
		In three dimensions, the embeddings $H^2(\om)\hookrightarrow L^\infty(\om)$ and $H^1(\om)\hookrightarrow L^6(\om)\hookrightarrow L^3(\om)$ together with \eqref{eUE} yield
		\begin{align*}
			\left|\e\ino\nabla\bF^\e:\nabla\bPS\right|
			&\leq \e\|\nabla\bF^\e\|_{L^2(\om)}\|\nabla\bPS\|_{L^2(\om)}
			\leq \sqrt{\e}\|\sqrt{\e}\nabla\bF^\e\|_{L^2(\om)}\|\bPS\|_{H^2(\om)},\\
			\left|\ino\bF^\e:(\bv^\e\cdot\nabla)\bPS\right|
			&\leq \|\bF^\e\|_{L^2(\om)}\|\bv^\e\|_{L^6(\om)}\|\nabla\bPS\|_{L^3(\om)}
			\leq C\|\bF^\e\|_{L^2(\om)}\|\bv^\e\|_{H^1(\om)}\|\bPS\|_{H^2(\om)},\\
			\left|\ino(q^\e-1)\bD(\bv^\e):\bPS\right|
			&\leq \|q^\e-1\|_{L^\infty(\om)}\|\bD(\bv^\e)\|_{L^2(\om)}\|\bPS\|_{L^2(\om)}
			\leq (c_3+1)\|\nabla\bv^\e\|_{L^2(\om)}\|\bPS\|_{H^2(\om)},\\
			\left|\ino(\bF^\e\bW(\bv^\e)-\bW(\bv^\e)\bF^\e):\bPS\right|
			&\leq 2\|\bF^\e\|_{L^2(\om)}\|\nabla\bv^\e\|_{L^2(\om)}\|\bPS\|_{L^\infty(\om)}
			\leq C\|\bF^\e\|_{L^2(\om)}\|\nabla\bv^\e\|_{L^2(\om)}\|\bPS\|_{H^2(\om)},\\
			\left|\alpha\ino\bF^\e:\bPS\right|
			&\leq \alpha\|\bF^\e\|_{L^2(\om)}\|\bPS\|_{L^2(\om)}
			\leq C\|\bF^\e\|_{L^2(\om)}\|\bPS\|_{H^2(\om)}.
		\end{align*}
		Therefore, we conclude
		\begin{multline*}
			\sup_{\|\bPS\|_{H^2(\om)}\leq 1}|\langle\partial_t\bF^\e,\bPS\rangle|
			\leq C\Bigl(\|\sqrt{\e}\nabla\bF^\e\|_{L^2(\om)}
			+ \|\bF^\e\|_{L^2(\om)}\|\bv^\e\|_{H^1(\om)}
			+ (c_3+1)\|\nabla\bv^\e\|_{L^2(\om)}\\
			+ \|\bF^\e\|_{L^2(\om)}\|\nabla\bv^\e\|_{L^2(\om)}
			+ \|\bF^\e\|_{L^2(\om)}\Bigr).
		\end{multline*}
		By \eqref{eUE} the right-hand side is bounded in $L^2(S)$; squaring and integrating over $S$ gives
		\begin{equation}\label{e_F_dual}
			\|\partial_t\bF^\e\|_{L^2(S;\,(H^2(\O;\bbS_0))^\ast)} \leq C,
		\end{equation}
		with $C$ independent of $\e$. Combining \eqref{eUE}--\eqref{e_F_dual} completes the proof.
	\end{proof}
	

	\begin{theorem}[Compactness and passage to the limit]\label{Compactness}
		Let $\{(\bv^\varepsilon, q^\varepsilon, \bF^\varepsilon)\}$ be a family of weak solutions of the 
		regularized problem \eqref{Rmodel+} satisfying the uniform estimates of Lemma~\ref{e_Ues}. 
		Then there exist functions
		\begin{equation}\label{lim_space}
			\bv \in L^\infty(0,T;L^2(\Omega)^3) \cap L^2(0,T;H_{0,\div}^1(\Omega)^3), \quad
			q \in L^\infty(\som), \quad
			\bF \in L^\infty(0,T;L^2(\Omega;\bbS_0))
		\end{equation}
		together with a subsequence, not relabeled, such that
		\begin{align}
			\label{evweak}
			&\bv^\e \wc \bv \quad \text{weakly in } L^2(0,T;H^1_0(\Omega)^3), 
			&\bv^\e \overset{*}{\wc} \bv \quad \text{weakly-$\ast$ in } L^\infty(0,T;L^2(\Omega)^3),\\
			\label{eFweak}
			&\bF^\e \wc \bF \quad \text{weakly in } L^2(0,T;L^2(\Omega;\bbS_0)), 
			&\bF^\e \overset{*}{\wc} \bF \quad \text{weakly-$\ast$ in } L^\infty(0,T;L^2(\Omega;\bbS_0)),\\
			\label{eqweak}
			&q^\e \wc q \quad \text{weakly in } L^2(0,T;L^2(\Omega)), 
			&q^\e \overset{*}{\wc} q \quad \text{weakly-$\ast$ in } L^\infty(0,T;L^2(\Omega)),\\
			\label{evstrong}
			&\bv^\varepsilon \sc \bv \quad \text{strongly in } L^2(0,T;L^p(\Omega)^3),\ 1\le p<6,
			&\bv^\e \wc \bv \quad \text{weakly in } L^2(0,T;L^6(\Omega)^3),\\
			\label{eFstrong-Hneg}
			&\bF^\e \sc \bF \quad \text{strongly in } L^2(0,T;(H^{1}(\Omega;\bbS_0))^\ast).
		\end{align}
		Moreover,
		\begin{equation}\label{eq:sqrt-eps-weak}
			\sqrt{\e}\,\nabla q^\e \rightharpoonup 0 \quad \text{weakly in } L^2(\som)^3, \qquad
			\sqrt{\e}\,\nabla \bF^\e \rightharpoonup 0 \quad \text{weakly in } L^2(\som;\bbS_0)^3 .
		\end{equation}
		In the tensor equation all terms admit a passage to the limit except the nonlinear terms
		$(q^\varepsilon - 1)\bD(\bv^\varepsilon)$ and $\bF^\varepsilon\bW(\bv^\varepsilon) - \bW(\bv^\varepsilon)\bF^\varepsilon$.
	\end{theorem}
	
	\begin{proof}
		\textbf{Step 1: Compactness.}
		
		The weak and weak-$\ast$ convergences \eqref{evweak}--\eqref{eqweak} follow, up to a subsequence, from the uniform bounds \eqref{U_v_bdd}--\eqref{U_F_bdd} by the Banach--Alaoglu theorem.
		
		The strong convergence \eqref{evstrong} follows from the Aubin--Lions--Simon lemma, since $\bv^\varepsilon$ is bounded in $L^2(S;H^1_{0,\div}(\Omega)^3)$, $\partial_t\bv^\varepsilon$ is bounded in $L^{4/3}(S;(H^1_{0,\div}(\Omega)^3)^*)$, and the embedding
		\[
		H^1_{0,\div}(\Omega)^3\hookrightarrow\hookrightarrow L^p(\Omega)^3,
		\qquad 1\le p<6,
		\]
		is compact.
		Similarly, \eqref{eFstrong-Hneg} follows from the Aubin--Lions--Simon lemma, since $\bF^\varepsilon$ is bounded in $L^2(S;L^2(\Omega;\bbS_0))$ and $\partial_t\bF^\varepsilon$ is bounded in $L^2(S;(H^{2}(\Omega;\bbS_0))^\ast)$, with
		\[
		L^2(\Omega;\bbS_0)\hookrightarrow\hookrightarrow (H^{1}(\Omega;\bbS_0))^\ast
		\hookrightarrow (H^{2}(\Omega;\bbS_0))^\ast.
		\]
		Finally, we prove the convergences \eqref{eq:sqrt-eps-weak}. Let
		$\bm{\phi}\in C_c^\infty(\som)^3$. Since $\bm{\phi}$ has compact support, no
		boundary term arises in the integration by parts
		\[
		\inso \sqrt{\e}\,\nabla q^\e\cdot\bm{\phi}
		= -\inso \sqrt{\e}\,q^\e\,\dv\bm{\phi} ,
		\]
		and the bound \eqref{U_q_bdd} gives
		$\|\sqrt{\e}\,q^\e\|_{L^2(\som)}\le\sqrt{\e}\,C\to0$, so the right-hand side
		vanishes as $\e\to0$. Hence $\sqrt{\e}\,\nabla q^\e\to0$ in
		$\mathcal{D}'(\som)$. Since \eqref{U_q_bdd} also bounds
		$\sqrt{\e}\,\nabla q^\e$ in $L^2(\som)^3$ uniformly in $\e$, and
		$C_c^\infty(\som)^3$ is dense in $L^2(\som)^3$, every weak limit point agrees
		with the distributional limit, whence $\sqrt{\e}\,\nabla q^\e\wc 0$ weakly in $L^2(\som)^3$.
		The same argument applied componentwise to $\bF^\e$, using \eqref{U_F_bdd},
		gives $\sqrt{\e}\,\nabla\bF^\e\wc 0$ weakly in
		$L^2(\som;\bbS_0)^3$. 
		
		By \eqref{q_bdd}, $\{q^\e\}$ is bounded in $L^\infty(\som)$ uniformly in $\e$.
		Hence, up to a further subsequence, $q^\e\overset{\ast}{\wc}q$ weakly-$\ast$
		in $L^\infty(\som)$. Since $\{g: c_2\le g\le c_3\text{ a.e.}\}$ is convex and
		closed, it is weakly-$\ast$ closed, so $c_2\le q\le c_3$ a.e. in $\som$.
		
		\textbf{Step 2: Limit in the momentum equation.}
		By the same argument as Step~6 of Theorem~\ref{Thm_existence}, using \eqref{e_v_dual}, 
		the strong convergence \eqref{evstrong}, and the weak convergence \eqref{eFweak}, 
		the limit $\varepsilon\to 0$ passes in \eqref{dweak1}. Therefore, $\bv$ satisfies the weak form of \eqref{Dmodel}$_1$.
		
		\textbf{Step 3: Limit in the scalar equation.}
		
		The dual estimate \eqref{e_q_dual} yields
		\[
		\int_S\langle\partial_t q^\varepsilon,\phi\rangle
		\longrightarrow
		\int_S\langle\partial_t q,\phi\rangle.
		\]
		Moreover, the estimate \eqref{U_q_bdd} implies that $\sqrt{\varepsilon}\,\nabla q^\varepsilon$ is bounded in $L^2(S;L^2(\Omega))^3,$ and therefore,
		\[
		\varepsilon\int_{S\times\Omega}
		\nabla q^\varepsilon\cdot\nabla\phi
		\longrightarrow0.
		\]
		For the transport term, we decompose
		\[
		\int_{S\times\Omega}
		q^\varepsilon\bm{v}^\varepsilon\cdot\nabla\phi
		=
		\int_{S\times\Omega}
		q^\varepsilon(\bm{v}^\varepsilon-\bm{v})\cdot\nabla\phi
		+
		\int_{S\times\Omega}
		q^\varepsilon\bm{v}\cdot\nabla\phi.
		\]
		The uniform bound \eqref{q_bdd}, together with the strong convergence \eqref{evstrong} and the weak-$\ast$ convergence of $q^\e$ in $L^\infty(\som)$ (note $\bv\cdot\nabla\phi\in L^1(\som)$), implies
		\[
		\int_{S\times\Omega}
		q^\varepsilon(\bm{v}^\varepsilon-\bm{v})\cdot\nabla\phi
		\longrightarrow0,
		\quad
		\int_{S\times\Omega}
		q^\varepsilon\bm{v}\cdot\nabla\phi
		\longrightarrow
		\int_{S\times\Omega}
		q\,\bm{v}\cdot\nabla\phi,
		\quad\text{and}\quad \int_{S\times\Omega}
		q^\varepsilon\phi
		\longrightarrow
		\int_{S\times\Omega}
		q\,\phi.
		\]
		Consequently, passing to the limit in the weak formulation of \eqref{dweak2} shows that $q$ satisfies \eqref{Dmodel}$_2$ in the weak sense.
		
		\textbf{Step 4: Limit in the tensor equation.}
		
		The uniform estimate \eqref{U_F_bdd} on $\partial_t\bm{F}^\varepsilon$, together with the strong convergence \eqref{eFstrong-Hneg}, implies
		\[
		\partial_t\bm{F}^\varepsilon
		\rightharpoonup
		\partial_t\bm{F}
		\quad\text{weakly in }
		L^2\big(S;(H^2(\Omega;\mathbb{S}_0))^*\big),
		\]
		where the limit is identified by integration by parts in time. Consequently,
		\[
		\int_S
		\langle\partial_t\bm{F}^\varepsilon,\bPS\rangle\,dt
		\longrightarrow
		\int_S
		\langle\partial_t\bm{F},\bPS\rangle\,dt,
		\qquad
		\forall\,\bPS\in L^2\big(S;H^2(\Omega;\mathbb{S}_0)\big).
		\] 
		Moreover, the Sobolev embedding $H^2(\Omega)\hookrightarrow W^{1,4}(\Omega)$
		yields $\nabla\bPS\in L^2(S;L^4(\Omega;\bbS_0))$. Together with the strong convergence \eqref{evstrong},
		\[
		\bm{v}^\varepsilon \sc \bm{v}
		\quad\text{strongly in }L^2(S;L^4(\Omega)^3),
		\]
		we obtain
		\[
		\bm{v}^\varepsilon\cdot\nabla\bPS
		\longrightarrow
		\bm{v}\cdot\nabla\bPS
		\quad\text{strongly in }L^1(S;L^2(\Omega;\bbS_0)).
		\]
		Therefore, using the weak-$\ast$ convergence \eqref{eFweak}, we conclude
		\[
		\int_{S\times\Omega}
		\bm{F}^\varepsilon:(\bm{v}^\varepsilon\cdot\nabla\bPS)
		\longrightarrow
		\int_{S\times\Omega}
		\bm{F}:(\bm{v}\cdot\nabla\bPS).
		\]
		In addition, the uniform estimate \eqref{U_F_bdd} gives $\sqrt{\e}\,\nabla\bF^\e$ bounded in $L^2(S;L^2(\Omega;\bbS_0))$, hence
		\[
		\varepsilon\int_{S\times\Omega}
		\nabla\bF^\varepsilon:\nabla\bPS\,dx\,dt
		\longrightarrow 0,
		\]
		while the weak convergence \eqref{eFweak} yields
		\[
		\int_{S\times\Omega}
		\bF^\varepsilon:\bPS\,dx\,dt
		\longrightarrow
		\int_{S\times\Omega}
		\bF:\bPS\,dx\,dt.
		\]
		Passing to the limit in the tensor equation, all terms can be identified except the nonlinear terms $$ (q^\e-1)\bD(\bv^\e)\quad\text{and}\quad \bF^\e\bW(\bv^\e)-\bW(\bv^\e)\bF^\e .$$ We therefore cannot yet conclude that $\bF$ satisfies \eqref{Dmodel}$_3$; the strong convergence of $q^\e$ is established next.
	\end{proof}
	
	
	\begin{theorem}
		\label{thm:qstrong}
		Let $(\bv^\e,q^\e,\bF^\e)$ be as in Lemma~\ref{e_Ues} and $(\bv,q)$ the limit in Theorem~\ref{Compactness}. Then  
		\begin{equation}\label{eqstrong}
			q^\e \sc q \quad \text{strongly in } L^2(\som).
		\end{equation}
		The limit $q\in L^\infty(S\X\O)\cap H^1(S;(H^{1}(\O))^\ast)$ is a weak solution of \eqref{Dmodel}$_2$, that is,
		\begin{equation}\label{q_eq}
			\ins \langle \partial_t q, \phi \rangle  - \int_{\som} q \bv \cdot \nabla \phi + \alpha \int_{\som} q \phi = 3\beta \int_{\som} \phi,
		\end{equation}
		for all $\phi \in L^2(0, T; H^1(\Omega))$. Moreover, for a.e.~$t \in (0, T)$,
		\begin{equation}\label{q_norm}
			\frac{1}{2} \|q(t)\|_{L^2(\Omega)}^2 + \alpha \int_0^t \|q(s)\|_{L^2(\Omega)}^2= \frac{1}{2} \|q_0\|_{L^2(\Omega)}^2 + 3\beta \int_{Q_t} q .
		\end{equation}
		In addition, $q^\e\to q$ strongly in $L^p(\som)$ for every finite $p \geq 1$.
	\end{theorem}
	\begin{proof}
		First, \eqref{q_eq} was obtained in Step~3 of the proof of Theorem~\ref{Compactness}.
		
		By \eqref{q_eq},
		$$\partial_t q =-\nabla\cdot( q\bv) - \alpha q +3\beta \in L^2(S;(H^1(\Omega))^\ast).$$ 
		
		Since the limit $q \in L^{\infty}(S \times \Omega) \cap H^{1}(S; (H^1(\Omega))^\ast)$
		does not in general belong to $L^{2}(S; H^{1}(\Omega))$, one cannot directly test the weak formulation \eqref{q_eq} by $q$. We therefore use the DiPerna--Lions regularization technique \cite{DiPernaLions1989}.
		
		Extend $q$ and $\bv$ by zero outside $\Omega$ and define the
		spatial mollification
		\begin{equation}\label{q_delta_def}
			q_{\delta}(t, x) := (q(t,\cdot) * \eta_{\delta})(x)
			= \int_{\R^3} q(t, y)\,\eta_{\delta}(x - y)\,dy,
			\qquad (t,x) \in S \times \R^3.
		\end{equation}
		Since $q \in L^{\infty}(S \times \Omega)$, standard mollification results give  
		\begin{align}
			q_{\delta} & \in L^{\infty} (S; W^{k,p}(\Omega))\quad \text{for all  $ k \geq 0,\, 1 \leq p \leq \infty$}, \quad \text{in particular $q_{\delta} \in L^{2}(S; H^{1}(\Omega))$}, \label{q_d}\\
			q_{\delta} &\sc q \quad \text{strongly in } L^{p}(S \times \Omega),
			\quad \text{for every finite } p \geq 1, \label{q_delta_conv}\\
			\quad q_{\delta}(t) & \sc q(t) \quad \text{strongly in } L^{p}(\Omega),
			\quad \text{for a.e. } t \in S. \label{q_delta_ptwise}
		\end{align}
		By \eqref{q_eq} tested with $\phi|_\Omega$, $\phi\in C_c^\infty(S\times\R^3)$,
		the extensions satisfy
		$\p_t\bar q+\operatorname{div}(\bar\bv\bar q)+\alpha\bar q=3\beta\chi_\Omega$
		in $\mathcal D'(S\times\R^3)$, where
		$\operatorname{div}(\bar\bv\bar q)=\bar\bv\cdot\nabla\bar q$ in the sense of
		\eqref{rdelta_dist} because $\operatorname{div}\bar\bv=\chi_\Omega\operatorname{div}\bv=0$.
		Mollifying this equation in the spatial variable yields
		\begin{equation}\label{q_delta_eq}
			\partial_{t}q_{\delta}
			+ \boldsymbol{v} \cdot \nabla q_{\delta}
			+ \alpha q_{\delta}
			= 3\beta(\chi_\Omega*\eta_\delta) 
			+ r_{\delta},\quad \text{for a.e.}\ t \in S,
		\end{equation}
		where the commutator error is defined by
		\begin{equation}\label{commutator_def}
			r_{\delta}
			:= \boldsymbol{v} \cdot \nabla q_{\delta}
			- \bigl(\boldsymbol{v} \cdot \nabla q\bigr) * \eta_{\delta}
			= \bigl(\boldsymbol{v} \cdot \nabla q_{\delta}\bigr)
			- \bigl(\boldsymbol{v} \cdot \nabla q\bigr)_{\delta}.
		\end{equation}
		Since $\bv\in L^2(S;H^1_0(\O)^3)$ and $q\in L^\infty(\som)$, Theorem~\ref{ap_rdelta} gives
		\begin{equation}\label{commutator_conv}
			r_{\delta} \sc 0
			\quad \text{strongly in } L^{1}(S; L^{2}(\Omega)),
			\quad \text{as } \delta \sc 0.
		\end{equation}
		Since $q_{\delta} \in L^{2}(S; H^{1}(\Omega))$ and
		$\partial_{t}q_{\delta} \in L^{2}(S; (H^1(\Omega))^\ast)$
		(because $\p_t q_\delta=(\p_t\bar q)*\eta_\delta$ and $\p_t\bar q\in L^2(S;H^{-1}(\R^3))$ by the extended equation),
		the Lions--Magenes lemma \cite{lionsMagenes1972} applies and yields
		\begin{equation}\label{LM_chain}
			\frac{1}{2}\frac{d}{dt}\|q_{\delta}(t)\|^{2}_{L^{2}(\Omega)}
			= \langle \partial_{t}q_{\delta},\, q_{\delta}
			\rangle_{(H^1(\Omega))^\ast,\, H^{1}(\Omega)},
			\qquad \text{in the distributional sense}.
		\end{equation}
		Using integration by parts, along with the incompressibility condition \eqref{Dmodel}$_4$ and boundary condition \eqref{Dmodel}$_6$, we obtain
		\begin{equation*}
			\int_{\Omega} (\boldsymbol{v} \cdot \nabla q_{\delta})\, q_{\delta}\,dx
			= \frac{1}{2}\int_{\Omega} \boldsymbol{v} \cdot \nabla(q_{\delta}^{2})\,dx
			= -\frac{1}{2}\int_{\Omega} (\nabla \cdot \boldsymbol{v})\, q_{\delta}^{2}\,dx = 0.
		\end{equation*}
		Testing \eqref{q_delta_eq} by $q_{\delta}$ over $Q_t$ yields
		\begin{equation}\label{energy_delta}
			\frac{1}{2}\|q_{\delta}(t)\|^{2}_{L^{2}(\Omega)}
			+ \alpha \int_{0}^{t}\|q_{\delta}(s)\|^{2}_{L^{2}(\Omega)}
			= \frac{1}{2}\|q_{\delta}(0)\|^{2}_{L^{2}(\Omega)}
			+3\beta\int_{Q_t}(\chi_\Omega*\eta_\delta)\,q_\delta
			+ \int_{0}^{t}
			\langle r_{\delta}(s),\, q_{\delta}(s)
			\rangle_{(H^1)^\ast,\, H^{1}}.
		\end{equation}
		We pass to the limit $\delta\to0$ in each term of \eqref{energy_delta}.
		
		Since $q_{\delta}(0) = q_{0} * \eta_{\delta}$, \eqref{q_delta_conv} gives
		$$\quad \|q_{\delta}(0)\|^{2}_{L^{2}(\Omega)}
		\sc \|q_{0}\|^{2}_{L^{2}(\Omega)}.$$
		Since $|q_{\delta}| \leq \|q\|_{L^{\infty}} < \infty$, the dominated convergence theorem together with \eqref{q_delta_ptwise} and \eqref{q_delta_conv} yields
		$$\int_{0}^{t}\|q_{\delta}(s)\|^{2}_{L^{2}(\Omega)}\,ds
		\sc \int_{0}^{t}\|q(s)\|^{2}_{L^{2}(\Omega)} \quad\text{and}\quad \|q_{\delta}(t)\|^{2}_{L^{2}(\Omega)}
		\sc \|q(t)\|^{2}_{L^{2}(\Omega)}
		\quad \text{for a.e.\ } t \in S.$$
		Also, $q_{\delta}$ is uniformly bounded in $L^{\infty}(S; L^{2}(\Omega))$, hence by \eqref{commutator_conv},
		$$\left|\int_{0}^{t}
		\langle r_{\delta}(s),\, q_{\delta}(s)\rangle_{(H^1)^\ast,\, H^{1}}\,ds\right|
		\leq  { \|r_{\delta}\|_{L^{1}(S; L^{2}(\Omega))}}
		\|q_{\delta}\|_{L^{\infty}(S; L^{2}(\Omega))} \sc 0.$$
		Moreover, since $\chi_\Omega*\eta_\delta\to\chi_\Omega$ in $L^1(\R^3)$ and
		$|q_\delta|\le\|q\|_{L^\infty}$, we have
		$3\beta\int_{Q_t}(\chi_\Omega*\eta_\delta)q_\delta\to3\beta\int_{Q_t}q$.
		Combining the above convergences, we pass to the limit $\delta \to 0$ in \eqref{energy_delta} to obtain \eqref{q_norm}.

		We prove \eqref{eqstrong}. Since $q^\e\wc q$ weakly in
		$L^2(\som)$, it suffices to prove convergence of the norms.
		
		Proceeding as in \eqref{En_q}, equation \eqref{dweak2} yields
		$$\frac{1}{2}\ddt\|q^\e\|^2_{L^2(\om)}+\e\|\nabla q^\e\|^2_{L^2(\om)}+\alpha\|q^\e\|^2_{\ltwo(\om)}= 3\beta \ino q^\e,\quad\text{for a.e. $t\in(0,T)$}.$$
		Integrating from $0$ to $t$, where $t\in[0,T]$, we obtain
		$$\frac{1}{2}\|q^\e(t)\|^2_{L^2(\om)}+\e\int_0^t\|\nabla q^\e\|^2_{L^2(\om)}+\alpha\int_0^t\|q^\e\|^2_{\ltwo(\om)}= \frac{1}{2}\|q_0\|^2_{L^2(\om)}+3\beta \int_{Q_t} q^\e .$$
		By \eqref{Rmodel+}$_5$, $q^\e(0)=q_0$ does not depend on $\e$, and the last term
		is linear in $q^\e$. We drop the nonnegative term $\e\int_0^t\|\nabla q^\e\|^2_{L^2(\om)}$
		and use \eqref{eqweak}. Comparing with \eqref{q_norm} then gives, for a.e.\ $t\in S$,
		\begin{equation}\label{limsup_q}
			\limsup_{\e\to 0}\Big(\frac{1}{2}\|q^\e(t)\|^2_{L^2(\om)}+\alpha\int_0^t\|q^\e\|^2_{\ltwo(\om)}\Big)
			\le \frac{1}{2}\|q_0\|^2_{L^2(\om)}+3\beta \int_{Q_t} q
			=\frac{1}{2} \|q(t)\|_{L^2(\Omega)}^2 + \alpha \int_0^t \|q\|_{L^2(\Omega)}^2 .
		\end{equation}
		By \eqref{U_q_bdd}, \eqref{e_q_dual} and the compact embedding
		$L^2(\om)\hookrightarrow\hookrightarrow(H^1(\om))^\ast$, the Aubin--Lions--Simon
		lemma gives $q^\e\to q$ in $C([0,T];(H^1(\om))^\ast)$. Together with the
		uniform bound in $L^\infty(S;L^2(\om))$, this yields $q^\e(t)\wc q(t)$ weakly
		in $L^2(\om)$ for every $t\in[0,T]$, so that
		$\|q(t)\|_{L^2(\om)}\le\liminf_{\e\to0}\|q^\e(t)\|_{L^2(\om)}$. Inserting this
		into \eqref{limsup_q} gives
		$$\limsup_{\e\to0}\|q^\e\|_{L^2(Q_t)}\le\|q\|_{L^2(Q_t)}\qquad\text{for a.e. } t\in S,$$
		and with the weak convergence \eqref{eqweak}, $q^\e\to q$ strongly in
		$L^2(Q_t)$ for a.e.\ $t\in S$. Finally, $|q^\e|,|q|\le c_3$ by \eqref{q_bdd} and Step 1 of Theorem~\ref{Compactness}, so
		$$\|q^\e-q\|^2_{L^2(\som)}\le\|q^\e-q\|^2_{L^2(Q_t)}+4c_3^2|\om|(T-t).$$
		Letting first $\e\to0$ and then $t\uparrow T$ along such $t$ gives \eqref{eqstrong}.

		Moreover, \eqref{eqstrong} and \eqref{eqweak} together with the uniform boundedness of $q^\e$ \eqref{q_bdd} imply $q^\e \sc q$ strongly in $L^p(\som)$ for every finite $p \geq 1$.
	\end{proof}

	\begin{theorem}[Defect-limit system]
		\label{thm:defect-limit}
		Assume that the triplet $(\bv^\e,q^\e,\bF^\e)$ satisfies the hypotheses of Theorem \ref{Compactness}.
		Let $(\bv,q,\bF)$ be the limit and subsequence given by Theorem~\ref{Compactness}.
		Then there exists a defect
		$$\bDc \in L^2\bigl(S;(H^2(\Omega;\bbS_0))^\ast\bigr)$$
		such that for all $\vp\in L^4(S;H_{0,\div}^1(\Omega)^3)$, $\phi\in L^2(S;H^1(\Omega))$, and $\bPS\in  L^2(S;H^2(\Omega;\bbS_0)),$
		the following hold:
		\begin{equation}\label{vlim}
			\rho \ins \langle \partial_t \bv, \vp \rangle - \rho \inso (\bv  \otimes \bv) : \nabla \vp + \mu_1 \inso \bD(\bv) : \bD(\vp) + \lambda \inso \bF : \bD(\vp) = \ins \langle \bff, \vp \rangle,
		\end{equation}
		\begin{equation}\label{qlim}
			\ins \langle \partial_t q, \phi \rangle  - \inso q \bv \cdot \nabla \phi + \alpha \inso q \phi = 3\beta \inso \phi,
		\end{equation}
		\begin{multline}\label{def_lim}
			\ins \langle\p_t\bF,\bPS\rangle - \inso \bF : (\bv \cdot \nabla)\bPS
			+ \inso (q - 1)\bD(\bv) : \bPS + \inso (\bF \bW(\bv) - \bW(\bv) \bF) : \bPS \\
			+\alpha \inso \bF : \bPS 
			+\int_0^T \langle \bDc(t),\bPS(t)\rangle_{(H^2)^\ast,H^2}\,dt
			=0.
		\end{multline}
	\end{theorem}
	\begin{proof}
		The identities \eqref{vlim} and \eqref{qlim} follow as in
		Theorem~\ref{Compactness}, so we prove only \eqref{def_lim}.
		
		\textbf{Step 1.} We prove 
		\begin{equation}\label{eqDv_lim}
			\inso (q^\e - 1)\bD(\bv^\e) :\bPS\sc \inso (q - 1)\bD(\bv) : \bPS
		\end{equation}
		We write
		$$\inso (q^\e - 1)\, \bD(\bv^\e) : \bPS \, dx\,dt=\inso \bD(\bv^\e) : \big((q^\e - 1)\bPS\big)\, dx\,dt.$$
		Let $\bPS\in L^2(S;H^2(\om;\bbS_0))$. By \eqref{embed_2},
		$\bPS\in L^2(S;L^\infty(\om))\subset L^2(\som)$. By \eqref{q_bdd} and
		Theorem~\ref{Compactness}, $|q^\e-q|\le 2c_3$ a.e., so
		$|(q^\e-q)\bPS|^2\le4c_3^2|\bPS|^2\in L^1(\som)$. By \eqref{eqstrong}, every
		subsequence of $(q^\e)$ has a further subsequence converging to $q$ a.e. in
		$\som$. Dominated convergence along it, together with the uniqueness of the
		limit, gives
		\[
		(q^\e-1)\bPS\sc(q-1)\bPS\qquad\text{strongly in }L^2(\som).
		\]
		Since $\bD(\bv^\e)\wc\bD(\bv)$ weakly in $L^2(\som)^{3\X3}$, \eqref{eqDv_lim} follows.

		\textbf{Step 2.} We introduce the defect term.
		
		We define the Jaumann commutator
		$$\bJc^\e := \bF^\e \bW(\bv^\e) - \bW(\bv^\e)\bF^\e.$$
		Since $\bF^\e$ is symmetric and $\bW(\bv^\e)$ is skew-symmetric, it follows that $\bJc^\e(x,t)\in\bbS_0$ for almost every $(t,x)\in S\times\Omega$. Next, H\"older's inequality and \eqref{U_F_bdd} yield
		$$\|\bJc^\e\|_{L^2(S;L^1(\Omega))} \leq C\|\bF^\e\|_{L^\infty(S;L^2(\Omega))} \|\nabla \bv^\e\|_{L^2(S;L^2(\Omega))} \leq C,$$
		where the constant $C>0$ is independent of $\e$.
		Therefore, using the continuous embedding \eqref{embed_2},
		$$\left|\inso \bJc^\e:\bPS \right|\leq C\|\bJc^\e\|_{L^2(S;L^1(\Omega))} \|\bPS\|_{L^2(S;L^\infty(\Omega))} \leq C\|\bPS\|_{L^2(S;H^2(\Omega))}.$$
		
		Thus, $\bJc^\e$ defines a bounded linear functional on
		$L^2(S;H^2(\Omega;\bbS_0))$, uniformly in $\e$. Equivalently, $\bJc^\e$ is bounded in $L^2\bigl(S;(H^2(\Omega;\bbS_0))^\ast\bigr).$ 
		Since $H^2(\Omega;\bbS_0)$ is a Hilbert space,
		its dual is also Hilbert. Therefore,
		$L^2\bigl(S;(H^2(\Omega;\bbS_0))^\ast\bigr)$ is reflexive.
		Hence, up to a subsequence, there exists
		\begin{equation}\label{J_weak}
			\bJc \in L^2\bigl(S;(H^2(\Omega;\bbS_0))^\ast\bigr)\quad\text{such that}\quad
			\bJc^\e \rightharpoonup \bJc
			\qquad\text{weakly in }L^2\bigl(S;(H^2(\Omega;\bbS_0))^\ast\bigr).
		\end{equation}
		Let us define
		$$\bDc:=\bJc-(\bF\bWv-\bWv\bF).$$
		Using the continuous embedding $H^2(\Omega;\bbS_0)\hookrightarrow L^\infty(\Omega;\bbS_0),$ we obtain the induced continuous embedding
		\[
		L^1(\Omega;\bbS_0)\hookrightarrow (H^2(\Omega;\bbS_0))^*.
		\]
		Therefore, using the regularity \eqref{lim_space} of $\bv$ and $\bF$, we deduce
		$$\bF\bWv-\bWv\bF\in L^2\bigl(S;(H^2(\Omega;\bbS_0))^\ast\bigr),$$
		and hence, \eqref{J_weak} yields
		$$\bDc\in L^2\bigl(S;(H^2(\Omega;\bbS_0))^\ast\bigr).$$
		Consequently,
		$$\inso \bJc^\e:\bPS \sc \ins \langle \bJc,\bPS \rangle_{(H^2)^\ast,H^2}= \int_{S\X\O} (\bF\bWv-\bWv\bF):\bPS+\ins \langle \bDc,\bPS \rangle_{(H^2)^\ast,H^2},$$
		for every $\bPS \in L^2(S; H^2( \Omega;\bbS_0))$.

		Finally, all the remaining terms in the weak formulation of the $\bF^\e$-equation pass to the limit exactly as in Theorem~\ref{Compactness}. 
		The only term for which strong
		compactness is not available is the Jaumann commutator, and the preceding argument shows that it converges to the defect functional $\bDc$ in the sense above. 
	\end{proof}
	
	\begin{remark}[Structure of the defect-limit system]
		\label{rem:strongform}
		Theorem~\ref{thm:defect-limit} produces a quadruplet $(\bv,q,\bF,\bDc)$
		satisfying \eqref{vlim}, \eqref{qlim} and the deviatoric equation with defect
		\begin{equation}
			\label{eq:strong-F-defect}
			\partial_{t}\bF
			+ (\bv\cdot\nabla)\bF
			+ \bF\bW(\bv)
			- \bW(\bv)\bF
			+ (q-1)\bD(\bv)
			+ \alpha\bF
			+ \bDc = 0,
		\end{equation}
		in the distributional sense specified in \eqref{def_lim}.
		
		The system \eqref{vlim}--\eqref{def_lim} has three equations for the four
		objects $(\bv, q, \bF, \bDc)$, but $\bDc$ is not a free unknown: it is fixed
		once the approximating sequence $(\bv^\e, q^\e, \bF^\e)$ and the subsequence
		are chosen. The system is therefore not self-contained: $\bDc$ cannot be determined from $(\bv, q, \bF)$ alone
		and carries information about the approximating sequence that survives in the
		limit only through this remainder. 
		In terms of Remark~\ref{rem:Geps}, the strictly local remodeling law is the
		closed limit of the weakly nonlocal laws as $\ell\to0$, along the chosen
		subsequence, exactly when $\bDc=0$.
		
		$\bDc$ records only the failure of the Jaumann commutator
		$\bF^\e\bW(\bv^\e)-\bW(\bv^\e)\bF^\e$ to pass to the limit under weak
		convergence of its two factors; the bound
		$\bDc\in L^2(S;(H^2(\om;\bbS_0))^\ast)$ comes from the $L^2(S;L^1(\om))$
		bound on the commutator and registers only its size. Whether the mechanism is
		oscillation or concentration is not decided at this level; see
		Section~\ref{sec-conclusion}.
	\end{remark}

	\begin{proposition}
		\label{cor:mu_zero_implies_m_zero}
		Assume the hypotheses of Theorem~\ref{thm:defect-limit}.
		If, in addition, either
		\begin{enumerate}
			\item $\boldsymbol{F}^{\varepsilon} \to \boldsymbol{F}$
			strongly in $L^{\infty}(0,T;L^{2}(\Omega;\bbS_0))$, or
			\item $\nabla\boldsymbol{v}^{\varepsilon} \to \nabla\boldsymbol{v}$
			strongly in $L^2(\som;\R^{3\X3})$,
		\end{enumerate}
		then the Jaumann defect
		\[
		\bDc=0
		\qquad\text{in }L^2\bigl(0,T;(H^2(\Omega;\bbS_0))^*\bigr).
		\]
	\end{proposition}
	\begin{proof}
		\textbf{Step 1:} First we assume $\bF^{\e} \to \bF$ strongly in $L^{\infty}(0,T;L^{2}(\Omega;\bbS_0))$. 
		
		Setting $\bW^\varepsilon := \bW(\bv^\e)$ and $\bW:=\bW(\bv)$, we write 
		\begin{align*}
			I= &\int_{S\times\Omega}
			(\bF^{\e} \bW^\varepsilon - \bW^\varepsilon \bF^{\e}):\bPS- (\bF\bW - \bW\bF):\bPS\\
			= &\underbrace{\int_{S\times\Omega}((\bF^{\e}-\bF)\bW^\varepsilon - \bW^\varepsilon(\bF^{\e}-\bF)):\bPS}_{I_1}+\underbrace{\int_{S\times\Omega}(\bF(\bW^\varepsilon-\bW) - (\bW^\varepsilon-\bW)\bF):\bPS}_{I_2}.
		\end{align*}
		For the first term, using \eqref{embed_2}, we obtain
		$$|I_1| \le
		2 \|\bF^{\e}-\bF\|_{L^\infty(S;L^2(\om))}
		\|\bW^\varepsilon\|_{L^2(S\times\Omega)} \|\bPS\|_{L^2(S;L^\infty(\om))} \sc 0.$$
		Next, using the matrix identity
		\[
		(\bF(\bW^\varepsilon-\bW) - (\bW^\varepsilon-\bW)\bF):\bPS
		=
		(\bW^\varepsilon-\bW):(\bF\bPS-\bPS \bF),
		\]
		together with $\bF\in L^\infty\big(S;L^2(\om;\bbS_0)\big)$ and
		$\bPS\in L^2\big(S;L^\infty(\om;\bbS_0)\big)$, we obtain
		$$\bF\bPS-\bPS \bF \in L^2(\som;\R^{3\X3}).$$
		Hence, the weak convergence $\bW^\varepsilon \rightharpoonup \bW\text{ weakly in }L^2(S\times\Omega;\mathbb R^{3\times 3}),$ gives $I_2\to0$. Therefore,
		\[
		\int_{S\times\Omega}
		\bigl(\bF^{\e} \bW(\bv^\e)-\bW(\bv^\e)\bF^{\e}\bigr):\bPS
		\sc
		\int_{S\times\Omega}
		\bigl(\bF\bW(\bv)-\bW(\bv)\bF\bigr):\bPS
		\qquad
		\text{for all }\bPS\in L^2(S;H^2(\O;\bbS_0)).
		\]
		Equivalently,
		\[
		\int_0^T \langle \bDc(t),\bPS(t)\rangle_{(H^2)^*,H^2}\,dt =0
		\qquad
		\text{for all }\bPS\in L^2(S;H^2(\O;\bbS_0)).
		\]
		Hence, we conclude that $\bDc=0\text{ in }L^2(0,T;(H^2(\Omega;\bbS_0))^*).$
		
		\textbf{Step 2.} Similarly, assume that $\nabla \bm{v}^\varepsilon\to \nabla \bm{v}$ $\text{strongly in } L^2(S\times\Omega;\mathbb R^{3\times3}).$
		It follows immediately that
		\[
		\bm{W}(\bm{v}^\varepsilon)\to \bm{W}(\bm{v})
		\qquad
		\text{strongly in }
		L^2(S\times\Omega;\mathbb R^{3\times3}),
		\]
		We now consider the difference
		$$I={
			\underbrace{ \int_{S\times\Omega}
				(\bm{F}^\varepsilon(\bm{W}(\bm{v}^\varepsilon)-\bm{W}(\bm{v}))-
				(\bm{W}(\bm{v}^\varepsilon)-\bm{W}(\bm{v}))\bm{F}^\varepsilon):\bm{\Psi}}_{I_1}
			+ \underbrace{\int_{S\times\Omega}
				((\bm{F}^\varepsilon-\bm{F})\bm{W}(\bm{v})-\bm{W}(\bm{v})(\bm{F}^\varepsilon-\bm{F})
				):\bm{\Psi}}_{I_2}.}$$
		Using H\"older's inequality, we deduce
		\begin{multline*}
			|I_1|\leq 2 \ins
			\|\bm{F}^\varepsilon(t)\|_{L^2(\Omega)} \|\bm{W}(\bm{v}^\varepsilon)-\bm{W}(\bm{v})\|_{L^2(\Omega)}\|\bm{\Psi}(t)\|_{L^\infty(\Omega)}\\
			\le 2
			\|\bm{F}^\varepsilon\|_{L^\infty(S;L^2(\Omega))}
			\|\bm{W}(\bm{v}^\varepsilon)-\bm{W}(\bm{v})\|_{L^2(S;L^2(\Omega))}
			\|\bm{\Psi}\|_{L^2(S;L^\infty(\Omega))} \sc 0.
		\end{multline*}
		Moreover, since $\bm{W}(\bm{v})\in L^2(S;L^2(\Omega;\R^{3\times 3}))\text{ and }
		\bm{\Psi}\in L^2(S;L^\infty(\Omega;\bbS_0))$, we obtain 
		$$(\bm{W}(\bm{v})\bm{\Psi}
		-\bm{\Psi}\bm{W}(\bm{v}))
		\in
		L^1(S;L^2(\Omega;\bbS_0)).$$ 
		Hence, arguing as in Step 1 and using the weak-$\ast$ convergence \eqref{eFweak}, we deduce that $I_2\to0$. Consequently,
		\[
		\int_{S\times\Omega}(\bF^{\e} \bW(\bv^\e)-\bW(\bv^\e)\bF^{\e}):\bPS\sc\int_{S\times\Omega}(\bF\bW(\bv)-\bW(\bv)\bF):\bPS,\qquad\text{for all }\bPS\in L^2(S;H^2(\O;\bbS_0)).
		\]
		Therefore, we conclude that $\bDc=0\text{ in }L^2(0,T;(H^2(\Omega;\bbS_0))^*)$, which completes the proof.
	\end{proof}
	
	\section{Conclusion and outlook}\label{sec-conclusion}
	
	The limit system of Theorem~\ref{thm:defect-limit} is closed except for the
	defect $\bDc$. We discuss three routes to remove it: additional regularity,
	an energy-level cancellation, or a broader solution concept. None is
	available at the regularity level of Theorem~\ref{thm:defect-limit}.
	
	\textbf{Removing the defect by additional regularity.}
	By Proposition \ref{cor:mu_zero_implies_m_zero} it would be enough to obtain
	strong convergence of one of the two factors of the Jaumann commutator. One
	may hope that the deviatoric equation propagates spatial regularity of
	$\bF^\e$ provided the velocity gradient is controlled in a strong enough norm,
	with the diffusive term entering a differentiated estimate with a favorable sign.
	This suggests a conditional criterion of Beale--Kato--Majda type,
	\begin{equation}\label{eq:criterion}
		\nabla\bv^\e\in L^1(0,T;L^\infty(\om))
		\quad\text{uniformly in }\e ,
	\end{equation}
	together with suitable regularity of the initial data.
	
	Turning this into an argument would require more than \eqref{eq:criterion}
	alone. The differentiated equation retains contributions from the stretching
	term and from the Jaumann commutator for which Lemma \ref{e_Ues} provides no
	$\e$-independent control; the associated integrations by parts are not
	justified on a general bounded Lipschitz domain under the Neumann condition
	\eqref{Rmodel+}$_6$; and Proposition \ref{cor:mu_zero_implies_m_zero} asks for
	convergence of $\bF^\e$ in $L^\infty(0,T;L^2(\om;\bbS_0))$, so the compactness
	argument would have to be arranged accordingly.
	
	The criterion \eqref{eq:criterion} is in any case not a consequence of the
	estimates of Lemma \ref{e_Ues}. In three dimensions it is not known to hold
	for Leray--Hopf solutions of the momentum equation, and establishing it for
	\eqref{Rmodel+} would mean developing a strong-solution theory, with bounds
	uniform in $\e$, for the coupled system rather than for the momentum equation
	alone, since the latter is forced by $\lambda\nabla\!\cdot\!\bF^\e$. Whether
	the small-data and short-time theories available for Oldroyd-type systems
	\cite{guillopeSaut1990,cheminMasmoudi2001} can be adapted in this way is not
	clear to us, and we do not pursue the two-dimensional case here either.
	
	\textbf{Why the energy cancellation does not seem to close.}
	A second possible route is an energy-level cancellation analogous to
	that of Lions and Masmoudi for the corotational Oldroyd model
	\cite{LionsMasmoudi2000}. There, the momentum and stress equations
	produce the same cross-defect
	\[
	\Theta :=
	\lim_{\e\to0}
	\left(
	\int_{S\times\Omega}\bF^\e:\bm D(\bv^\e)
	-
	\int_{S\times\Omega}\bF:\bm D(\bv)
	\right),
	\]
	with opposite weights, allowing cancellation. In the present system,
	however, the momentum equation
	\eqref{Rmodel+}$_1$ couples through $\lambda\bF^\e$ with a constant
	coefficient, whereas the deviatoric equation \eqref{Rmodel+}$_3$ couples through
	$(q^\e-1)\bm D(\bv^\e)$. The corresponding cross-defect is therefore
	\[
	\Theta_q :=
	\lim_{\e\to0}
	\left(
	\int_{S\times\Omega}
	(q^\e-1)\bm D(\bv^\e):\bF^\e
	-
	\int_{S\times\Omega}
	(q-1)\bm D(\bv):\bF
	\right),
	\]
	which is distinct from $\Theta$ and cannot be canceled by the same
	energy combination.
	
	Although $q^\e\to q$ strongly in every
	$L^p(S\times\Omega)$, $p<\infty$, the remainder
	\[
	\int_{S\times\Omega}
	(q^\e-q)\bm D(\bv^\e):\bF^\e
	\]
	need not vanish because $\bm D(\bv^\e):\bF^\e$ is controlled only
	in $L^1(S\times\Omega)$. Even if this remainder vanished, the remaining
	$q$-weighted cross-defect would not be controlled by $\Theta$ for
	nonconstant $q$, and its treatment would require sufficient
	regularity to pair the defect with the limiting deviatoric field. The obstruction therefore comes from the
	stretching term $(\operatorname{tr}\bE-1)\bm D(\bv)$ in the
	morphoelastic strain law~\eqref{eq:intro-strain}, whose coefficient depends on the volumetric
	remodeling state and has no counterpart in the viscoelastic fluid
	models for which the cancellation argument was developed.
	
	\textbf{Widening the solution concept.}
	The third route abandons the distributional formulation. As recalled in
	Section~\ref{sec:intro}, Eiter, Hopf, and Lasarzik \cite{eiterHopfLasarzik2023}
	meet the failure of weak sequential compactness of the Jaumann commutator in
	the same form as here, and replace the weak formulation by a relative-energy
	inequality, obtaining energy-variational solutions in the spirit of the
	dissipative solutions of Lions \cite{lions1996}.
	
	That formulation is closed: the limit objects satisfy a self-contained
	solution concept with weak-strong uniqueness, and no undetermined remainder
	survives. The price is that the stress evolution is no longer satisfied in the
	sense of distributions, so that the relation between the limit stress and the
	original model is mediated by the relative energy rather than by the equation
	itself. The defect formulation makes the opposite choice: it keeps the
	distributional equation, at the price of a term that the analysis does not
	determine, and isolates that term rather than absorbing it. Which is
	preferable depends on what one wants from the limit. Applying the relative
	energy method to \eqref{Dmodel} would require adapting it to the coupling
	$(q-1)\bDv$, whose variable coefficient is absent from the models treated in
	\cite{eiterHopfMielke2022,eiterHopfLasarzik2023}. Whether the pointwise bounds
	on $q$ from Lemma \ref{lem:qbdd} are of use here we leave open.
	
	\textbf{Outlook.}
	Two questions remain open within the present framework. The first is whether
	$\bDc$ can be nonzero. The limit system is not closed, but we exhibit no
	sequence along which the defect fails to vanish. For the corotational Oldroyd
	system, Lions and Masmoudi \cite[Sec.~4.2]{LionsMasmoudi2000} show by an
	explicit oscillating construction that the Jaumann commutator need not be
	weakly continuous, and one may ask whether an analogous construction is
	possible here. The situation differs, since the coefficient of the stretching
	term is itself an unknown and $q^\e$ converges strongly, so any oscillation
	would have to be carried by $\bF^\e$ and $\nabla\bv^\e$ alone while remaining
	compatible with the transport of the trace. Constructing such a sequence, or
	excluding it by an argument outside the energy estimates, would settle the
	status of the limit system. The second question is whether the defect admits a
	finer representation. The bound
	$\bDc\in L^2\bigl(0,T;(H^2(\om;\bbS_0))^*\bigr)$ obtained here comes from the
	$L^2(S;L^1(\om))$ bound on the commutator and records only its size; a
	Young-measure description would separate oscillation from concentration in the
	sense of \cite{diPernaMajda1987}, while the $H$-measures of Tartar and
	G\'erard \cite{tartar1990,gerard1991} would in addition resolve the directions
	in which $\bF^\e$ and $\bW(\bv^\e)$ oscillate against one another. Since the
	defect arises from a product of two weakly convergent sequences with no shared
	compactness, a microlocal description of that interaction is the natural
	setting for the first question as well.
	
	\section*{Acknowledgement}
	Swarupananda Banerjee acknowledges IIT Kharagpur for funding his PhD position. This work was largely completed when Amartya Chakrabortty was a PhD student and research assistant in the Department of Processes and Material Simulation, Fraunhofer ITWM, Kaiserslautern, Germany.
	
	\section*{Declarations}
	\textbf{Competing interests.} The authors declare no competing interests.\\
	\textbf{Data availability.} Data sharing is not applicable to this article, as no datasets were generated or analyzed.

	\appendix

	
	\section{Auxiliary results}
	\subsection{Maximum principle for the scalar equation}
	\begin{lemma}
		\label{lem:qbdd}
		Let $(\bv^\e,q^\e)$ be a weak solution of the truncated system\footnote{That is, Definition~\ref{Def_Weak} with $q^\e$ replaced by $\mathscr T(q^\e)$ in the transport and coupling terms; equivalently, the limit of \eqref{dgweak1}--\eqref{dgweak3} obtained in Step 8 of Theorem~\ref{Thm_existence}.}, where $\alpha>0$ and $\beta \geq 0$.
		If the initial datum satisfies
		\begin{equation}\label{AssR01}
			0 < c_0 \le q_0(x) \le c_1, \quad \text{in } \om,
		\end{equation}
		then
		\begin{equation}\label{q_bdd}
			0<c_2\leq q^\e(x,t)\leq c_3,\quad\text{in } \som ,
		\end{equation}
		with constants $c_2,c_3$ independent of $\e$.
	\end{lemma}
	
	\begin{proof}
		Since $(\bv^\e,q^\e)$ is a weak solution of the truncated problem, it satisfies
		\begin{equation}\label{truncated_weak_q}
			\left\langle \partial_t q^\e,\phi\right\rangle
			+\e\int_\Omega\nabla q^\e\cdot\nabla\phi
			-\int_\Omega
			\mathscr{T}(q^\e)\,
			\bv^\e\cdot\nabla\phi
			+\alpha\int_\Omega q^\e\phi=3\beta\int_\Omega\phi,
		\end{equation}
		for every test function $\phi\in H^1(\om)$. Moreover, $q^\e\in L^2(S;H^1(\om))$ and $\partial_tq^\e\in L^2(S;(H^1(\om))^\ast)$.
		
		Define the spatially homogeneous barriers
		\begin{equation}\label{barriers}
			\underline{q}(t):=\frac{3\beta}{\alpha}+\Big(c_0-\frac{3\beta}{\alpha}\Big)e^{-\alpha t},
			\qquad
			\overline{q}(t):=\frac{3\beta}{\alpha}+\Big(c_1-\frac{3\beta}{\alpha}\Big)e^{-\alpha t}.
		\end{equation}
		Both solve the relaxation equation $y'+\alpha y=3\beta$ and are independent of $x$. Let $b$ denote either of them. Then $\nabla b=0$, so the diffusion term drops, and $\mathscr T(b(t))$ is constant in space, so that
		\[
		\int_\Omega \mathscr T(b)\,\bv^\e\cdot\nabla\phi
		=\mathscr T(b)\int_\Omega \bv^\e\cdot\nabla\phi=0
		\qquad\text{for all }\phi\in H^1(\om),
		\]
		by the incompressibility condition together with $\bv^\e=0$ on $\partial\om$. Hence $b$ satisfies \eqref{truncated_weak_q} for every $\phi\in H^1(\om)$.
		
		Let $\Phi(s):=\int_0^s\mathscr T(\tau)\,d\tau$ be the primitive introduced in Step 3 of Theorem~\ref{Thm_existence}, so that $\Phi\in C^{1,1}(\mathbb R)$ with $\Phi'=\mathscr T$. Fix $t\in S$ and let $b=b(t)$ be spatially constant. Since $\min\{q^\e(t),b\}\in H^1(\om)$ and $\Phi$ is Lipschitz, the chain rule for Sobolev functions gives
		\[
		\nabla\,\Phi\big(\min\{q^\e,b\}\big)
		=\mathscr T(q^\e)\,\chi_{\{q^\e<b\}}\nabla q^\e
		\qquad\text{a.e. in }\om .
		\]
		Therefore, setting $z:=b-q^\e$ and $z^+:=\max\{z,0\}$, we have $\nabla z^+=-\chi_{\{q^\e<b\}}\nabla q^\e$, and
		\begin{equation}\label{transport_vanishes}
			\int_\Omega \mathscr T(q^\e)\,\bv^\e\cdot\nabla z^+
			=-\int_\Omega \bv^\e\cdot\nabla\,\Phi\big(\min\{q^\e,b\}\big)
			=\int_\Omega(\nabla\cdot\bv^\e)\,\Phi\big(\min\{q^\e,b\}\big)=0,
		\end{equation}
		the boundary term vanishing because $\bv^\e=0$ on $\partial\om$. The same computation with $\max\{q^\e,b\}$ in place of $\min\{q^\e,b\}$ gives
		\begin{equation}\label{transport_vanishes_upper}
			\int_\Omega \mathscr T(q^\e)\,\bv^\e\cdot\nabla w^+=0,
			\qquad w^+:=(q^\e-b)^+ .
		\end{equation}
		Take $b=\underline q$ and $z:=\underline q-q^\e$. The map $r\mapsto r^+$ is Lipschitz continuous, so $z^+\in L^2(S;H^1(\om))$ is an admissible test function, and since $\partial_t z\in L^2(S;(H^1(\om))^\ast)$,
		\[
		\left\langle \partial_t z,z^+\right\rangle_{(H^1(\om))^\ast,H^1(\om)}
		=\frac12\frac{d}{dt}\|z^+\|_{L^2(\om)}^2 .
		\]
		As shown above, the barrier $\underline q$ satisfies \eqref{truncated_weak_q}. Subtracting \eqref{truncated_weak_q} for $q^\e$ from \eqref{truncated_weak_q} for $\underline q$, testing by $z^+$, and using $\nabla\underline q=0$ together with the vanishing of the $\mathscr T(\underline q)$ contribution, we obtain
		\[
		\frac12\frac{d}{dt}\|z^+\|_{L^2(\om)}^2
		+\e\|\nabla z^+\|_{L^2(\om)}^2
		+\alpha\|z^+\|_{L^2(\om)}^2
		+\int_\Omega \mathscr T(q^\e)\,\bv^\e\cdot\nabla z^+=0 .
		\]
		The last term vanishes by \eqref{transport_vanishes}, so that
		\[
		\frac12\frac{d}{dt}\|z^+\|_{L^2(\om)}^2
		+\e\|\nabla z^+\|_{L^2(\om)}^2
		+\alpha\|z^+\|_{L^2(\om)}^2=0 .
		\]
		Since $q_0\ge c_0=\underline q(0)$, we have $z^+(\cdot,0)=(c_0-q_0)^+=0$, hence $z^+=0$ a.e. in $\som$, that is, $q^\e\ge\underline q$.
		
		Take $b=\overline q$ and $w:=q^\e-\overline q$. Repeating the argument with $w^+$ in place of $z^+$ and using \eqref{transport_vanishes_upper}, we obtain
		\[
		\frac12\frac{d}{dt}\|w^+\|_{L^2(\om)}^2
		+\e\|\nabla w^+\|_{L^2(\om)}^2
		+\alpha\|w^+\|_{L^2(\om)}^2=0 .
		\]
		Since $q_0\le c_1=\overline q(0)$, we have $w^+(\cdot,0)=0$, hence $w^+=0$ a.e. in $\som$ and $q^\e\le\overline q$. The two bounds together give
		\begin{equation}\label{q_bdd_time}
			\underline{q}(t)\le q^\e(x,t)\le \overline{q}(t),
			\qquad\text{a.e. in } \som .
		\end{equation}
		Both $\underline q$ and $\overline q$ are monotone in $t$ and take values between their endpoints. For $\beta>0$,
		\[
		\underline q(t)\ge\min\Big\{c_0,\frac{3\beta}{\alpha}\Big\},
		\qquad
		\overline q(t)\le\max\Big\{c_1,\frac{3\beta}{\alpha}\Big\},
		\qquad t\in[0,T],
		\]
		whereas for $\beta=0$ we have $\underline q(t)=c_0e^{-\alpha t}\ge c_0e^{-\alpha T}>0$ and $\overline q(t)=c_1e^{-\alpha t}\le c_1$. In either case \eqref{q_bdd} follows from \eqref{q_bdd_time} with the
		constants $c_2,c_3$ of Step~1.1 of Theorem~\ref{Thm_existence}, which depend
		only on $c_0,c_1,\alpha,\beta,T$ and not on $\e$.
	\end{proof}

	\subsection{DiPerna--Lions commutator lemma}
	We recall \cite[Lemma~II.1]{DiPernaLions1989}.
	\begin{lemma}[Lemma II.1, \cite{DiPernaLions1989}]
		\label{ap_Diperna}
		Let $p'$ denote the conjugate exponent of $p$, i.e. $\frac{1}{p}+\frac1{p'}=1$, and let
		$s$ satisfy
		$$\text{for } r \text{ or } p<\infty,\quad \frac{1}{s}=\frac{1}{r}+\frac{1}{p},
		\qquad\text{for }r=p=\infty,\quad s<\infty \text{ arbitrary}.$$
		\begin{enumerate}
			\item Let $\bm B\in (W^{1,r}_{\mathrm{loc}}(\mathbb{R}^N))^N$,
			$w\in L^p_{\mathrm{loc}}(\mathbb{R}^N)$ with $1\le p\le\infty$ and $r\ge p'$. Then
			\[
			(\bm B\cdot\nabla w)*\eta_\delta-\bm B\cdot\nabla(w*\eta_\delta)
			\xrightarrow{\delta\to0}0\quad\text{in }L^s_{\mathrm{loc}}(\mathbb{R}^N).
			\]
			\item Let $\bm B\in L^1(0,T;(W^{1,r}_{\mathrm{loc}}(\mathbb{R}^N))^N)$,
			$w\in L^\infty(0,T;L^p_{\mathrm{loc}}(\mathbb{R}^N))$. Then
			\begin{equation}\label{ap_Diperna_2}
				(\bm B\cdot\nabla w)*\eta_\delta-\bm B\cdot\nabla(w*\eta_\delta)
				\xrightarrow{\delta\to0}0\quad\text{in }L^1(0,T;L^s_{\mathrm{loc}}(\mathbb{R}^N)).
			\end{equation}
		\end{enumerate}
	\end{lemma}
	
	\begin{theorem}[Global commutator convergence on bounded domains]\label{ap_rdelta}
		Let $\Omega\subset\mathbb{R}^3$ be a bounded Lipschitz domain and $T>0$. Suppose $\bm v\in L^2(0,T;H^1_0(\Omega)^3)$ and $ q\in L^\infty(S\times\Omega).$
		Define the commutator $r_\delta$ on $S\times\Omega$ by
		\begin{equation}\label{rdelta_def}
			r_\delta:=\bm v\cdot\nabla(q*\eta_\delta)-(\bm v\cdot\nabla q)*\eta_\delta,
		\end{equation}
		where $(\bm v\cdot\nabla q)*\eta_\delta$ is understood distributionally through
		\begin{equation}\label{rdelta_dist}
			\bm v\cdot\nabla q\equiv\operatorname{div}(\bm vq)-q\operatorname{div}\bm v .
		\end{equation}
		Then, we have the following convergence result
		\[
		\lim_{\delta\to0}\|r_\delta\|_{L^1(0,T;L^2(\Omega))}=0 .
		\]
	\end{theorem}
	
	\begin{proof}
		Let $\bar \bv$ and $\bar q$ denote the extensions by zero onto
		$S\times\mathbb{R}^3$:
		\[
		\bar \bv(x,t)=
		\begin{cases}
			\bv(x,t), & x\in\Omega,\\
			0, & \text{otherwise},
		\end{cases}
		\qquad
		\bar q(x,t)=
		\begin{cases}
			q(x,t), & x\in\Omega,\\
			0, & \text{otherwise}.
		\end{cases}
		\]
		The extension by zero maps $H^1_0(\Omega)^3$ into $H^1(\mathbb{R}^3)^3$ with
		$\nabla\bar{\bm v}=\chi_\Omega\nabla\bm v$; indeed, for
		$\varphi\in C_c^\infty(\mathbb{R}^3)^3$,
		\[
		\int_{\mathbb{R}^3}\bar{\bm v}\operatorname{div}\varphi
		=\int_\Omega \bm v\operatorname{div}\varphi
		=-\int_\Omega\nabla\bm v:\varphi
		=-\int_{\mathbb{R}^3}(\chi_\Omega\nabla\bm v):\varphi.
		\]
		No boundary term arises because $\bv$ has zero trace. Hence,
		\begin{equation}\label{vq_bar}
			\bar{\bm v}\in L^2(0,T;H^1(\mathbb{R}^3)^3)\qquad\text{and}\qquad
			\|\bar q\|_{L^\infty(S\times\mathbb{R}^3)}=\|q\|_{L^\infty(S\times\Omega)} .
		\end{equation}
		Since \(S=(0,T)\) has finite measure and the global spaces are continuously embedded in their respective local spaces, it follows that
		$$\bar{\bm v}\in L^1(0,T;H^{1}_{\mathrm{loc}}(\mathbb{R}^3)^3)\qquad\text{and}\qquad \bar q\in L^\infty(0,T;L^\infty_{\mathrm{loc}}(\mathbb{R}^3)).$$
		Since $\bm v\in H^1_0(\Omega)^3$, the trace of $\bm v$ on $\partial\Omega$
		vanishes, so for every $\vp\in C_c^\infty(\mathbb{R}^3)$ and a.e.\ $t\in S$, we deduce
		\[
		\big\langle\operatorname{div}(\bar{\bm v}\bar q),\vp\big\rangle
		=-\int_{\mathbb{R}^3}\bar q\,\bar{\bm v}\cdot\nabla\vp
		=-\int_{\Omega}q\,\bm v\cdot\nabla\vp
		=\big\langle\chi_\Omega\operatorname{div}(\bm vq),\vp\big\rangle.
		\]
		Hence, in the sense of distributions, $\operatorname{div}(\bar \bv\bar q)
		=\chi_\Omega\,\operatorname{div}(\bv q)$. Similarly, $\bar q\,\operatorname{div}\bar \bv=\chi_\Omega q\,\operatorname{div}\bv.$ Therefore, using the distributional definition of the transport term,
		\[
		\bar \bv\cdot\nabla\bar q
		:=\operatorname{div}(\bar \bv\bar q)-\bar q\,\operatorname{div}\bar \bv
		=\chi_\Omega\bigl(\operatorname{div}(\bv q)-q\,\operatorname{div}\bv\bigr) 
		=\chi_\Omega(\bv\cdot\nabla q).
		\] 
		Consequently, with
		\[
		\bar r_\delta:=\bar{\bm v}\cdot\nabla(\bar q*\eta_\delta)
		-\operatorname{div}(\bar{\bm v}\bar q)*\eta_\delta
		+(\bar q\operatorname{div}\bar{\bm v})*\eta_\delta ,
		\]
		the identity \eqref{rdelta_dist} and $(\bar q*\eta_\delta)|_\Omega=q*\eta_\delta$
		give
		\begin{equation}\label{restriction}
			\bar r_\delta\big|_{S\times\Omega}=r_\delta .
		\end{equation}
		Apply Lemma~\ref{ap_Diperna}\,\eqref{ap_Diperna_2} with $\bm B=\bar{\bm v}$, $w=\bar q$, $r=2$ and $p=\infty$. The conjugate exponent is $p'=1$, so $r\ge p'$ holds, and since
		$r<\infty$,
		\[
		\frac{1}{s}=\frac{1}{r}+\frac{1}{p}=\frac12,\qquad\text{i.e. }s=2 .
		\]
		The hypotheses of the lemma hold by \eqref{vq_bar}; therefore,
		\begin{equation}\label{bar_rdelta}
			\bar r_\delta\xrightarrow{\delta\to0}0\qquad\text{strongly in }
			L^1(0,T;L^2_{\mathrm{loc}}(\mathbb{R}^3)).
		\end{equation}
		Since $\Omega$ is bounded, there exists a bounded open set $K\supset\overline{\Omega}$. Then \eqref{bar_rdelta} gives
		$$\|\bar r_\delta\|_{L^1(0,T;L^2(K))}\sc 0,$$
		and by \eqref{restriction}, we conclude
		\[
		\|r_\delta\|_{L^1(0,T;L^2(\Omega))}
		=\|\bar r_\delta\|_{L^1(0,T;L^2(\Omega))}
		\le\|\bar r_\delta\|_{L^1(0,T;L^2(K))}\longrightarrow0 .
		\]
		This completes the proof.
	\end{proof}
	
	\addcontentsline{toc}{section}{References}
	\small	
	\bibliographystyle{plain}
	\bibliography{references}

\end{document}